\documentclass[11pt,a4paper]{article}

\usepackage{tikz}
\usepackage{subfigure}
\usepackage[english]{babel}
\usepackage{graphicx}
\usepackage{extarrows}

\usepackage[center]{caption2}
\usepackage{amsfonts,amssymb,amsmath,latexsym,amsthm}
\usepackage{multirow}
\DeclareSymbolFont{bbold}{U}{bbold}{m}{n}
\DeclareSymbolFontAlphabet{\mathbbold}{bbold}

\usepackage{ebezier,eepic}
\usepackage{color}
\usepackage{multirow}

\usepackage{epsf,epsfig,amsfonts,amsgen,amsmath,amstext,amsbsy,amsopn,amsthm,amsfonts,amssymb,amscd
}
\usepackage{ebezier,eepic}
\usepackage{color}
\usepackage{multirow}
\def\va{\vec{a}}
\def\ex{\mathrm{ex}}
\def\biex{\mathrm{biex}}
\def\geqs{\geq}
\def\leqs{\leq}
\def\N{\mathcal{N}}

\def\K{K^{(r)}_{s,t}}

\newtheorem{thm}{Theorem}[section]
\newtheorem{cor}[thm]{Corollary}
\newtheorem{lem}[thm]{Lemma}
\newtheorem{prop}[thm]{Proposition}

\newtheorem{claim}{Claim}[section]

\newtheorem{Definition}{Definition}[section]

\begin{document}

\title{Some sharp results on the generalized Tur\'an numbers}
\author{
Jie Ma\footnote{School of Mathematical Sciences,
University of Science and Technology of China, Hefei, 230026,
P.R. China. Email: jiema@ustc.edu.cn. Research partially supported by National Natural Science Foundation of China (NSFC)
grants 11501539 and 11622110.}
~~~~~~~~
Yu Qiu\footnote{School of Mathematical Sciences,
University of Science and Technology of China, Hefei, 230026,
P.R. China. Email: yuqiu@mail.ustc.edu.cn.}
}

\date{}

\maketitle

\begin{abstract}
For graphs $T, H$, let $\ex(n,T,H)$ denote the maximum number of copies of $T$ in an $n$-vertex $H$-free graph.
In this paper we prove some sharp results on this generalization of Tur\'an numbers, where our focus is for the graphs $T,H$ satisfying $\chi(T)<\chi(H)$.
This can be dated back to Erd\H{o}s \cite{Erdos-62}, where he generalized the celebrated Tur\'an's theorem by showing that for any $r\geq m$,
the Tur\'an graph $T_r(n)$ uniquely attains $\ex(n,K_m,K_{r+1})$.
For general graphs $H$ with $\chi(H)=r+1>m$, Alon and Shikhelman \cite{Alon-2015} showed that $\ex(n,K_m,H)=\binom{r}{m}(\frac{n}{r})^m+o(n^m)$.
Here we determine this error term $o(n^m)$ up to a constant factor.
We prove that $\ex(n,K_m,H)=\binom{r}{m}(\frac{n}{r})^m+\biex(n,H)\cdot\Theta(n^{m-2})$, where $\biex(n,H)$ is the Tur\'an number of the decomposition family of $H$.
As a special case, we extend Erd\H{o}s' result, by showing that $T_r(n)$ uniquely attains $\ex(n,K_m,H)$
for any edge-critical graph $H$.
We also consider $T$ being non-clique, where even the simplest case seems to be intricate.
Following from a more general result, we show that for all $s\leq t$, $T_2(n)$ maximizes the number of $K_{s,t}$ in $n$-vertex triangle-free graphs if and only if $t<s+\frac12+\sqrt{2s+\frac14}$.
\end{abstract}

\section{Introduction}
Let $T$ and $H$ be two fixed graphs. Throughout the paper we denote by $\N(G,T)$ the number of copies of $T$ in a graph $G$,
and let $\ex(n,T,H)$ be the maximum number of copies of $T$ in an $n$-vertex $H$-free graph.

The well-known Tur\'an's theorem \cite{Turan-41} states that the maximum number of edges in an $n$-vertex $K_{r+1}$-free graph is uniquely attained by the {\it Tur\'an graph} $T_r(n)$,
i.e., the complete balanced $r$-partite graph on $n$ vertices.
This was generalized by Erd\H{o}s \cite{Erdos-62} as following.

\begin{thm}[\cite{Erdos-62}]\label{THM: Erdos complete}
For all $n\geq r\geq m\geq 2$, the Tur\'an graph $T_r(n)$ uniquely attains the maximum number of cliques $K_m$ in an $n$-vertex $K_{r+1}$-free graph.
\end{thm}
\noindent Since then the function $\ex(n,T,H)$ for $T\neq K_2$ was studied for certain pairs $\{T,H\}$
(such as \cite{BG,Grz,GL,HHKNR}; see \cite{Alon-2015} for an elaborated discussion).
This was culminated in \cite{Alon-2015} by Alon and Shikhelman, where they systematically studied the function $\ex(n,T,H)$.
Among other results, they \cite{Alon-2015} proved that for any graph $H$ with chromatic number $\chi(H)=r+1>m$,
\begin{equation}\label{equ:AS}
\ex(n,K_m,H)=\N(T_r(n),K_m)+o(n^m).
\end{equation}
Recently this function has been the subject of extensive research,
including \cite{AKS,FKL,GGMV,GMV,GKPP,GS,LTTZ,Luo,MYZ} (by no means a comprehensive list).

In this paper we determine the error term $o(n^m)$ in \eqref{equ:AS} up to a constant factor.
Given a graph $H$ with $\chi(H)=r+1$, the \emph{decomposition family} of $H$, denoted by $\mathcal F_H$,
is the family of all bipartite graphs that are obtained from $H$ by deleting $r-1$ color classes in some $(r+1)$-coloring of $H$.
By $\biex(n,H)$ we denote the maximum number of edges in an $n$-vertex graph
which does not contain any graph in $\mathcal F_H$ as a subgraph.
Our main result is as following.
\begin{thm}\label{THM: Main Theorem}
For any integer $m$ and any graph $H$ with $\chi(H)=r+1>m\geq 2$, $$\ex(n,K_m,H)=\N(T_r(n),K_m)+\biex(n,H)\cdot\Theta(n^{m-2}).$$
\end{thm}
\noindent Since $\biex(n,H)=O(n^{2-\alpha_H})$ for some $\alpha_H>0$ by the classic result of K\"ov\'ari, Tur\'an and S\'os \cite{Sos-54},
this improves the error term to $O(n^{m-\alpha_H})$.
A graph is {\it edge-critical} if there exists some edge whose deletion reduces its chromatic number.
Simonovits \cite{Simonovits-66} proved that for any edge-critical graph $H$ with $\chi(H)=r+1\geq 3$ and for sufficiently large $n$,
the Tur\'an graph $T_r(n)$ is the unique graph which attains the maximum number of edges in an $n$-vertex $H$-free graph.
It is clear that if $H$ is edge-critical, then $\biex(n,H)=0$.
This enables us to obtain the following
\begin{cor}\label{cor:edge-critical}
Let $H$ be an edge-critical graph with $\chi(H)=r+1>m\geq 2$ and $n$ be sufficiently large.
Then the Tur\'an graph $T_r(n)$ is the unique graph attaining the maximum number of $K_m$'s in an $n$-vertex $H$-free graph.
\end{cor}
\noindent This can be viewed as a common generalization of the result of
Erd\H{o}s \cite{Erdos-62} and the result of Simonovits \cite{Simonovits-66}.
To prove the upper bound of Theorem \ref{THM: Main Theorem}, we establish a stability result.

\begin{thm}\label{THM: partition}
Let $H$ be a graph with $\chi(H)=r+1> m\geq 2$.
If $G$ is an $n$-vertex $H$-free graph with $\N(G,K_m)\geqs\N(T_r(n),K_m)-o(n^m)$,
then $G$ can be obtained from $T_r(n)$ by adding and deleting a set of $o(n^2)$ edges.
\end{thm}

Let us prove the lower bound of Theorem \ref{THM: Main Theorem} here.
The construction we use can be found in \cite{Allen-2014}.
Let $F''$ be an $n$-vertex $\mathcal F_H$-free graph with $\biex(n,H)$ edges.
Then $F''$ contains an $\lceil\frac nr\rceil$-vertex subgraph $F'$ with at least $\biex(n,H)/2r^2$ edges.
One can further find a bipartite subgraph $F$ of $F'$ with at least $\biex(n,H)/4r^2$ edges.
Consider the graph $G$ obtained by inserting $F$ into the largest part of $T_r(n)$.
Since $\chi(H)=r+1$ and $F$ is bipartite, it's easy to see $G$ is $H$-free.
As each edge in $F$ is contained in $\Omega(n^{m-2})$ copies of $K_m$ in $G$,
we see that $\ex(n,K_m,H)\geqs\N(G,K_m)\geqs\N(T_r(n),K_m)+\biex(n,H)\cdot \Omega(n^{m-2})$.

We also consider the function $\ex(n,T,H)$ for some $T$ not being a clique.
In the next result we maximize the number of some complete $r$-partite graphs $T$ in $K_{r+1}$-free graphs.
It reveals that the relatively sizes of the parts in $T$ will play an important role.
\begin{thm}\label{cor:mulpart}
(i) Let $n$ be sufficiently large and $T$ be any complete balanced $r$-partite graph. Then the Tur\'an graph $T_r(n)$ is the unique $n$-vertex $K_{r+1}$-free graph which maximizes the number of $T$-copies. \\
(ii) Let $n$ be sufficiently large and $t\geq s$. Then $T_2(n)$ maximizes the number of copies of $K_{s,t}$ in $n$-vertex triangle-free graphs if and only if $t<s+\frac12+\sqrt{2s+\frac14}$.
\end{thm}
\noindent This will follow from Theorem \ref{THM: MAIN ex(n,K_{s:(r-1);t})} in Section \ref{sec:rpartite} in a more general setting.

The remaining of this paper is organized as follows.
In Section 2 we give out some preliminaries.
In Section 3 we prove Theorem \ref{THM: partition}.
The proofs of Theorem \ref{THM: Main Theorem} and Corollary \ref{cor:edge-critical} will be completed in Section 4.
In Sections 5 and 6, we show Theorem \ref{THM: MAIN ex(n,K_{s:(r-1);t})}, which would imply Theorem \ref{cor:mulpart}.
Section 7 contains some concluding remarks.
Throughout the paper, let $[k]=\{1,\cdots,k\}$ for a positive integer $k$.

\section{Preliminaries}
In this section we will present some definitions and results needed in the subsequent sections.

Let $\sigma(H)$ be the smallest size of a color class in a proper $\chi(H)$-coloring of a graph $H$.
So if $H$ is edge-critical, then $\sigma(H)=1$.
The next proposition can be found in \cite{Allen-2014}; we include its short proof for the completeness.
\begin{prop}[\cite{Allen-2014}]\label{Obs: sigma VS biex} If $H$ is a graph with $\chi(H)\geq 3$ and $\sigma(H)\geqs2$, then $\biex(n,H)\geqs n-1$.
\end{prop}
\begin{proof}
We have $\sigma(H)\geq 2$. Then any $F\in\mathcal F_H$ contains a matching of size $2$.
So $K_{1,n-1}$ must be $\mathcal F_H$-free, implying that $\biex(n,H)\geq e(K_{1,n-1})=n-1$.
\end{proof}

Next we collect some properties on the counts of cliques in Tur\'an graphs $T_r(n)$.

\begin{prop}\label{prop:turan1}
For any integers $n\ge r\geq s\geq m\geq 2$, it holds that
$$\N(T_r(n), K_m)\ge \N(T_s(n), K_m) \text{~~~and~~~} \N(T_r(n), K_m)=\binom{r}{m}\left(\frac{n}{r}\right)^m+O(n^{m-1}).$$
\end{prop}

Fix a graph $H$ and consider a graph $G$.
For each $v\in V(G)$, let $d_G(v,H)$ denote the number of copies of $H$ in $G$ containing the vertex $v$, and let $\delta(G,H)=\min_{x\in V}d_G(x,H)$.
If $H=K_m$, then we write $d_G(v,K_m)$ and $\delta(G,K_m)$ as $d_G^{(m)}(v)$ and $\delta^{(m)}(G)$, respectively.

\begin{prop}\label{prop:turan2}
	For any integers $n-1\geqs r\geqs m\geqs2$, it holds that
	\begin{eqnarray*}
	\delta^{(m)}(T_r(n))=\N(T_r(n),K_m)-\N(T_r(n-1),K_m).
	\end{eqnarray*}
\end{prop}
\begin{proof}
Let $V_1,V_2,\cdots,V_r$ be the partition classes of $T_r(n)$. Then for any $v\in V_i$ with $|V_i|=\lceil\frac nr\rceil$, we have $T_r(n-1)=T_r(n)-\{v\}$ and thus
	\begin{equation*}
	d^{(m)}(v)=\N(T_r(n),K_m)-\N(T_r(n-1),K_m).
	\end{equation*}
It then suffices to show that $d^{(m)}(v)=\delta^{(m)}(T_r(n))$.
Suppose to the contrary that $d^{(m)}(v)>\delta^{(m)}(T_r(n))$ for some $v\in V_i$.
Then there exists a vertex $u\in V_j$ with $d^{(m)}(u)=\delta^{(m)}(T_r(n))<d^{(m)}(v)$.
Then we must have $|V_j|=\lfloor\frac nr\rfloor<\lceil\frac nr\rceil$.
Thus the graph $G'$ obtained from $T_r(n)$ by deleting the vertex $u$ is not $T_{r}(n-1)$.
Since $\N(G',K_m)=\N(T_r(n),K_m)-d^{(m)}(u)$, it follows that
	\begin{eqnarray*}
	\N(G',K_m)>\N(T_r(n),K_m)-d^{(m)}(v)=\N(T_r(n-1),K_m).
	\end{eqnarray*}
This contradicts Theorem \ref{THM: Erdos complete}, completing the proof.
\end{proof}

The \emph{clique number} of a graph $G$, denoted by $\omega(G)$, is the maximum size of a clique in $G$.
We will use a result due to Eckhoff \cite{Eckhoff-2004}.
\begin{thm}[\cite{Eckhoff-2004}]\label{THM: key technique}
Let $G$ be an $n$-vertex graph with the clique number $\omega:=\omega(G)\geq m\geq 2$.
Let $n_1$ and $n_2$ be the unique integers satisfying that $e(G)=e(T_\omega(n_1))+n_2$ and $0\leq n_2<\frac{\omega-1}{\omega}n_1.$
Then, $\N(G,K_m)\leqs\N(T_\omega(n_1),K_m)+\N(T_{\omega-1}(n_2),K_{m-1}).$
\end{thm}
\noindent Note that in the setting we have $n_1\leq n$. To see this, we notice that as $G$ is $K_{\omega+1}$-free, it follows by $e(T_\omega(n_1))\le e(G)\le e(T_\omega(n))$.

The following structural stability theorem was originally proved by Erd\H{o}s and Simonovits \cite{Erd66,Erdos-66,ES66,Simonovits-66}
(also see F\"uredi \cite{Furedi-2015} for a new proof in the case of $H$ being cliques).

\begin{thm}[Erd\H{o}s-Simonovits Stability Theorem]\label{THM:simonovits stability}
Let $H$ be a graph with $\chi(H)=r+1\geqs3$. Then, for every $\varepsilon>0$, there exist $\delta=\delta(H,\varepsilon)>0$ and $n_0=n_0(H,\varepsilon)\in\mathbb{N}$
such that the following holds. If $G$ is an $H$-free graph on $n\geqs{n_0}$ vertices with $e(G)\geq e(T_r(n))-\delta{n^2}$,
then there exists a partition of $V(G)=V_1\dot{\cup}\cdots\dot{\cup}V_{r}$ such that $\sum^{r}_{i=1}e(V_i)<\varepsilon n^2/2$.
Therefore, $G$ can be obtained from $T_r(n)$ by adding and deleting a set of at most $\varepsilon n^2$ edges.
\end{thm}

A classical result of Andr\'asfai, Erd\H{o}s and S\'os \cite{Andrasfai-Erdos-Sos-74} asserts that a $K_{r+1}$-free graph with large minimum degree must be $r$-partite.

\begin{thm}[\cite{Andrasfai-Erdos-Sos-74}]\label{LEM: Andrasfai-Erdos-Sos}
Let $n>r\geq 2$. If $G$ is a $K_{r+1}$-free graph on $n$ vertices with $\delta(G)>\frac{3r-4}{3r-1}n$, then $G$ is $r$-partite.
\end{thm}

We need the celebrated Szemer\'edi's regularity lemma \cite{Szemeredi-78}.
Let $X,Y$ be disjoint subsets in a graph $G$.
By $G[X,Y]$ we denote the bipartite subgraph of $G$ consisting of all edges that has one endpoint in $X$ and another in $Y$; let $e_G(X,Y)$ (respectively, $e_G(X)$) be the number of edges in $G[X,Y]$ (respectively, in $G[X]$).
For mutually disjoint $V_1,\cdots,V_r\subseteq V(G)$, similarly we define $G[V_1,\cdots,V_r]$ to be the $r$-partite subgraph of $G$
consisting of all edges in $\cup_{1\leqs i<j\leqs r}E(G[V_i,V_j])$.
The subscripts will be dropped if there is no confusion.
The \emph{density} of the pair $(X,Y)$ is defined by $d(X,Y):=e_G(X,Y)/|X||Y|$.
The pair $(X,Y)$ is called $\varepsilon$-\emph{regular} if $|d(X,Y)-d(A,B)|<\varepsilon$ for all $A\subseteq X$ and $B\subseteq Y$ with $|A|\geq\varepsilon|X|$ and $|B|\geqs\varepsilon|Y|$.
A partition $V_0,\cdots,V_k$ of $V$ is $\varepsilon$-\emph{regular}, if
$|V_0|\leqs\varepsilon|V|$, $|V_1|=\cdots=|V_k|$, and all but at most $\varepsilon k^2$ of pairs $(V_i,V_j)$ with $1\leqs i<j\leqs k$ are $\varepsilon$-regular.

\begin{thm}[Regularity Lemma]\label{THM:regularity lemma}
	For every $\varepsilon>0$, there exists $M=M(\varepsilon)$ such that every graph of order at least $\varepsilon^{-1}$ admits an $\varepsilon$-regular partition $\{V_0,\cdots,V_k\}$ with $\varepsilon^{-1}\leqs k\leqs M$.
\end{thm}

For a real $d\in (0,1]$, an $\varepsilon$-regular pair $(X,Y)$ is called {\it $(\varepsilon,d)$-regular} if the density $d(X,Y)\ge d$.
Given an $\varepsilon$-regular partition $\{V_0,\cdots,V_k\}$ of a graph $G$,
the $(\varepsilon,d)$-\emph{cluster graph} is a graph $R$ with the vertex set $V(R)=[k]$ and with edges $ij\in E(R)$ if and only if $(V_i,V_j)$ is an $(\varepsilon,d)$-regular pair.
For an integer $s\geq 1$, the $s$-\emph{blowup} of $G$, denoted by $G(s)$, is the graph obtained from $G$ by
replacing every vertex $v\in V(G)$ with an independent set $I_v$ of size $s$ and replacing every edge $uv\in E(G)$ with the complete bipartite graph between $I_u$ and $I_v$.
Let $\Delta(G)$ be the maximum degree of $G$.

\begin{thm}[Embedding Lemma; see \cite{Diestel-3rd}]\label{LEM: Cluster blowup}
For all $d\in(0,1]$ and $\Delta\geqs1$ there exists a $\gamma_0>0$ with the following property.
If a graph $G$ has a $\gamma$-regular partition $\{V_0,\cdots,V_k\}$ with $|V_1|=\cdots=|V_k|=\ell$ and the $(\gamma,d)$-cluster graph $R$,
where $\gamma\leqs\gamma_0$ and $\ell d^\Delta\geqs2s$ for some integer $s\geq 1$,
then any subgraph $H$ of the $s$-blowup of $R$ with $\Delta(H)\leq\Delta$ is also a subgraph of $G$.
\end{thm}

\section{A stability result on the number of cliques}
In this section we prove Theorem \ref{THM: partition}, which is restated as the following.
\begin{thm}\label{THM: partition2}
For any $\varepsilon>0$, integers $r\geq m\geq 2$ and a fixed graph $H$ with $\chi(H)=r+1$, there exist $\delta=\delta(H,\varepsilon)>0$ and $n_0=n_0(H,\varepsilon)\in\mathbb N$ such that the following holds.
Let $G$ be an $H$-free graph on $n\geq n_0$ vertices with $\N(G,K_m)\geq\N(T_r(n),K_m)-\delta n^m$.
Then $G$ can be obtained from $T_r(n)$ by adding and deleting a set of at most $\varepsilon n^2$ edges.
\end{thm}

We first establish a lemma, which says that it will be enough to find a partition of $V(G)$ into $r$ parts such that the number of edges contained in a part is at most $o(n^2)$.

\begin{lem}\label{LEM:InEdges}
Let $H, G$ be from Theorem \ref{THM: partition2} and $\varepsilon\gg\eta \gg \delta\gg 1/n_0$.\footnote{Throughout this paper, the notation $\varepsilon_1\gg \varepsilon_2$ simply means that $\varepsilon_2$ is a sufficiently small function of $\varepsilon_1$ which is needed to satisfy some inequalities in the proof.}
If $V_1, \cdots, V_r$ is a partition of $V(G)$ with $\sum_{i=1}^r e(V_i)< \eta n^2$,
then $e(G[V_1,\cdots,V_r])>e(T_r(n))-\varepsilon n^2$.
\end{lem}
\begin{proof}
Let $G'=G[V_1,\cdots,V_r]$. So $\omega:=\omega(G')\leq r$.
Every $K_m$-copy in $G$ either contains some edge in $\cup_{i=1}^{r}E(G[V_i])$ or is contained in $G'$.
Since $\sum_{i=1}^r e(V_i)< \eta n^2$, the number of $K_m$-copies of the former type is at most $\eta n^m$. So we have
$\N(G,K_m)\leqs\N(G',K_m)+\eta n^m.$

Let $n_1,n_2$ be the unique integers satisfying that $e(G')=e(T_\omega(n_1))+n_2$ and $0\leq n_2<\frac{\omega-1}{\omega}n_1$.
If $\omega<m$, then $\N(G',K_m)=0$ and thus $\N(G, K_m)\le \varepsilon n^m$, a contradiction.
So $\omega\ge m$. Then by Theorem \ref{THM: key technique},
$$\N(G',K_m)\leqs\N(T_\omega(n_1),K_m)+\N(T_{\omega-1}(n_2),K_{m-1}).$$
We also have $n_2< n_1\leq n$ and thus $\N(T_{\omega-1}(n_2),K_{m-1})\leq n^{m-1}\leq \eta n^m$.
Now combining the above inequalities, we have
$$\N(T_r(n),K_m)-\delta n^m\leq \N(G,K_m)\leq \N(T_\omega(n_1),K_m)+2\eta n^m,$$
where the first inequality is given by the conditions.
Since $\omega\le r$, $n_1\leq n$ and $\varepsilon\gg\eta \gg \delta\gg 1/n$,
it yields $\omega=r$ and $n_1>(1-\varepsilon)n.$
By the definition of $n_1$, we can conclude that
$$e(G')\geqs e(T_r(n_1))>e(T_r(n))-\varepsilon n^2.$$
This completes the proof of the lemma.
\end{proof}

Now we are ready to prove Theorem \ref{THM: partition2}.

\begin{proof}[Proof of Theorem \ref{THM: partition2}.]
We are given $\varepsilon>0$ and a fixed graph $H$ with $\chi(H)=r+1>m\geq 2$.
We will choose the constants appeared in this proof satisfying the following hierarchy:
\begin{equation}\label{equ:hierarchy}
\varepsilon\gg\eta \gg \delta\gg 1/k_0\gg \gamma_0\gg 1/n_0,
\end{equation}
where $\eta$ is from Lemma \ref{LEM:InEdges} and each of $\delta, k_0, \gamma_0, n_0$ can be expressed as functions of $H, \varepsilon, \eta$ and the previous constants in this order.
Let $G$ be an $H$-free graph on $n\geqs n_0$ vertices with
\begin{equation}\label{equ:NKm}
\N(G,K_m)\geq\N(T_r(n),K_m)-\delta n^m\geq \binom{r}{m}\left(\frac nr\right)^m-2\delta n^m,
\end{equation}
where the last inequality follows by Proposition \ref{prop:turan1}.
We will show that
\begin{equation}\label{equ:InEdges}
\text{there exists a partition of } V(G)=V_1\dot{\cup}\cdots\dot{\cup}V_r
\text{ such that } \sum_{i=1}^{r}e(V_i)<\eta n^2.
\end{equation}
Note that by Lemma \ref{LEM:InEdges} and Theorem \ref{THM:simonovits stability},
this would imply that $G$ can be obtained from $T_r(n)$ by adding and deleting a set of at most $\varepsilon n^2$ edges.

Let $d:=\frac{\delta}{2}$ and $\Delta:=\Delta(H)$. Then there exists a real $\gamma_0>0$
such that the conclusion of Lemma \ref{LEM: Cluster blowup} holds for $d$ and $\Delta$,
and in addition, $\gamma_0$ satisfies the hierarchy \eqref{equ:hierarchy}.
By Theorem \ref{THM:regularity lemma}, there exists a $\gamma_0$-regular partition $\mathcal{A}:=\{A_0,\cdots,A_k\}$ of $G$ with $\gamma_0^{-1}\leq k\leq M(\gamma_0)$.
Let $\ell=|A_1|=\cdots=|A_k|$. As $|A_0|<\gamma_0 n$, we have $\ell\geq\frac{1-\gamma_0}{k}n\geq \frac{1-\gamma_0}{M(\gamma_0)}n_0$ and thus
we can choose $n_0$ so that $\ell d^{\Delta}\geq 2|V(H)|$.
Let $R$ be the $(\gamma_0,d)$-cluster graph of $\mathcal{A}$.

We first show that the clique number $\omega:=\omega(R)$ is at most $r$.
Suppose for a contradiction that $K_{r+1}\subseteq R$.
Then $H\subseteq K_{r+1}(|V(H)|)\subseteq R(|V(H)|)$, which, together with Lemma \ref{LEM: Cluster blowup}, implies that $H\subseteq G$, a contradiction.
Thus $R$ is a $K_{r+1}$-free graph on $k\ge \gamma_0^{-1}$ vertices.

The following claim gives an estimation on the number of edges in $R$.

\medskip

{\noindent\bf{Claim.}}  $e(R)\ge e(T_r(k))-c_H\delta^{1/m} k^2$, where $c_H>0$ is a constant only depending on $H$.

\medskip

\noindent {\it Proof of the claim.}
Let $n_1$ and $n_2$ be the unique integers satisfying $e(R)=e(T_\omega(n_1))+n_2$ and $0\leq n_2<\frac{\omega-1}{\omega}n_1$.
By Theorem \ref{THM: key technique} and its remark, $n_2<n_1\le |V(R)|=k$ and $$\N(R,K_m)\leqs \N(T_\omega(n_1),K_m)+\N(T_{\omega-1}(n_2),K_{m-1}).$$
Since $\omega\le r$ and $\N(T_{\omega-1}(n_2),K_{m-1})\leq k^{m-1}$, by Proposition \ref{prop:turan1}, we have
\begin{equation*}
\N(R,K_m)\leq{\binom{r}{m}}\left(\frac{n_1}r\right)^m+O(k^{m-1}).
\end{equation*}
By the choices of $\gamma_0$ and $k$, we have $\frac{1}{k}\leq \frac{1}{k_0}\ll \delta$, implying that
\begin{equation}\label{equ: N(R,Km)<}
\N(R,K_m)\leq{\binom{r}{m}}\left(\frac{n_1}r\right)^m+\delta\cdot k^m.
\end{equation}

We then estimate the number $\N(G, K_m)$ of the copies of $K_m$ in $G$, which must belong to one of the following five types.
For those copies of $K_m$ containing some vertex in $A_0$, since $|A_0|<\gamma_0 n$, these copies will contribute no more than $\gamma_0 n^m$ to $\N(G, K_m)$.
For those copies of $K_m$ containing at least two vertices in $A_i$ for some $i\in[k]$, since $\gamma_0^{-1}\leq k$ and $k\ell\leq n$, they will contribute at most $ k\binom{\ell}{2}n^{m-2}\leq\gamma_0 n^m$.
For those copies of $K_m$ containing some edge in non-$\gamma_0$-regular pairs of $\mathcal{A}$, since there are at most $\ell^2\cdot\gamma_0 k^2$ such edges,
there are at most $\ell^2\gamma_0 k^2n^{m-2}\leq\gamma_0 n^m$ such copies.
For those copies of $K_m$ containing some edge in $\gamma_0$-regular pairs of $\mathcal{A}$ with density $<d$, since there are at most $d\ell^2\binom{k}{2}$ such edges,
there are at most $d\ell^2\binom{k}{2}n^{m-2}\leq d n^m$ such copies.
For those copies of $K_m$ not belonging to the above types, all of their edges must be in $(\gamma_0,d)$-regular pairs of $\mathcal{A}$,
and thus there are at most $\N(R,K_m)\cdot \ell^m\leq {\N(R,K_m)}\cdot\left(\frac{n}{k}\right)^m$ such copies of $K_m$.
Summing up these five types, we have, as $3\gamma_0+d\leq\delta$, that
\begin{equation*}\label{equ: N(G,Km)<}
\N(G,K_m)\leq 3\gamma_0 n^m+d n^m+\frac{\N(R,K_m)}{k^m}n^m\leqs\delta n^m+\frac{\N(R,K_m)}{k^m}n^m.
\end{equation*}
Together with \eqref{equ:NKm}, this implies that
\begin{equation}\label{equ: N(R,Km)>}
\frac{\N(R,K_m)}{k^m}\geq\binom{r}{m}\left(\frac 1r\right)^m-3\delta.
\end{equation}
Combining with \eqref{equ: N(R,Km)<} and \eqref{equ: N(R,Km)>}, we have
\begin{equation*}
\frac{n_1}{k}\geq 1-c\cdot \delta^{1/m},
\end{equation*}
where the constant $c>0$ depends on $r$ and $m$ (and thus only depends on $H$).
By the definition of $n_1$, we have
\begin{equation*}\label{equa: e(TrN1)>}
e(R)\geqs e(T_r(n_1))\geqs e(T_r(k))-c_H\delta^{1/m}\cdot k^2,
\end{equation*}
completing the proof of the claim. \qed

\medskip

We now choose $\delta=\delta(\eta,K_{r+1})$ and $k_0=k_0(\eta,K_{r+1})$ according to Theorem \ref{THM:simonovits stability} such that
for any $K_{r+1}$-free graph $\mathcal{G}$ on $k\geq k_0$ vertices with $e(\mathcal{G})\ge e(T_r(k))-c_H\delta^{1/m}\cdot k^2$,
there exists a partition of $V(\mathcal{G})= W_1\dot{\cup}\cdots\dot{\cup} W_r$ such that $\sum_{i=1}^{r}e_{\mathcal{G}}(W_i)<\eta k^2/2$.

We have seen that the cluster graph $R$ is a $K_{r+1}$-free graph on $k\ge \gamma_0^{-1}\ge k_0$ vertices.
Therefore, by the claim, there is a partition of $V(R)=[k]=W_1\dot{\cup}\cdots\dot{\cup} W_r$ such that $$\sum_{i=1}^{r}e_R(W_i)<\eta k^2/2.$$
Then one can partition $V(G)$ into the following $r$ parts: $V_1=(\cup_{j\in W_1}A_j)\cup A_0$ and $V_i=\cup_{j\in W_i}A_j$ for $i\in \{2,\cdots,r\}$.
It remains to estimate the number of edges in $\cup_{i=1}^{r}G[V_i]$, each of which belongs to one of following five types:
(Note that $d=\delta/2$, $k\ell\leq n$ and $\frac{1}{k}\leq \gamma_0$.)
\begin{itemize}
\item[-] edges incident to some vertex in $A_0$, the number of which is at most $\gamma_0 n^2$,
\item[-] edges in $G[A_i]$ for some $i\in [k]$, the number of which is at most $ k\binom{\ell}{2}\leq \gamma_0 n^2$,
\item[-] edges in non-$\gamma_0$-regular pairs of $\mathcal{A}$, the number of which is at most $\ell^2\gamma_0 k^2\leq \gamma_0 n^2$,
\item[-] edges in $\gamma_0$-regular pairs $\mathcal{A}$ with density $<d$, the number of which is at most $d\ell^2k^2\leq \delta n^2/2$, and
\item[-] edges in some $(\gamma_0,d)$-regular pair $(A_{j_1},A_{j_2})$ for some $j_1,j_2\in W_i$ and $i\in[r]$, the number of which is at most $\ell^2\sum_{i=1}^{r}e_R(W_i)<\ell^2\eta k^2/2\leq\eta n^2/2$.
\end{itemize}
Combining, as $3\gamma_0+\delta/2+\eta/2<\eta$, we have that $\sum_{i=1}^{r}e(V_i)<\eta n^2.$
This proves \eqref{equ:InEdges} and thus completes the proof of Theorem \ref{THM: partition2}.
\end{proof}

\section{Counting cliques}
This section will be devoted to the proof of Theorem \ref{THM: Main Theorem} (from which Corollary \ref{cor:edge-critical} will also follow).
We have established the lower bound. So it suffices to show for sufficiently large $n$, if $G$ is an $n$-vertex $H$-free graph with
\begin{equation}\label{equa: N(G,Km)>N(Trn,Km)}
\N(G,K_m)\geq \N(T_r(n),K_m),
\end{equation}
then
\begin{equation}\label{equ:N(G,Km)}
\N(G,K_m)\leq \N(T_r(n),K_m)+\biex(n,H)\cdot O(n^{m-2}).
\end{equation}

We will proceed with a sequence of claims.
\begin{claim}\label{LEM: large deltam}
We may assume an additional condition for $G$ that $\delta^{(m)}(G)\geq\delta^{(m)}(T_r(n))$.
\end{claim}
\begin{proof}
Assume $n\geq n_0+\binom{n_0}{m}$ for some sufficiently large $n_0$.
Let $G_n:=G$. If $\delta^{(m)}(G_n)\geq \delta^{(m)}(T_r(n))$, then there is nothing to show.
So we may assume there exists some vertex $v_n\in V(G_n)$ with $d_{G_n}^{(m)}(v_n)\leq \delta^{(m)}(T_r(n))-1$. Let $G_{n-1}:=G_n-\{v_n\}$.
Then, by Proposition \ref{prop:turan2}, we have $\N(G_{n-1},K_m)=\N(G_n,K_m)-d_{G_n}^{(m)}(v_n)\geq\N(T_r(n),K_m)-\delta^{(m)}(T_r(n))+1=\N(T_r(n-1),K_m)+1$.
	
We then iteratively define graphs $G_j$ satisfying $\N(G_j,K_m)\geq \N(T_r(j),K_m)+(n-j)$ as following.
Assume that $G_j$ is defined. If there exists some $v_j\in G_j$ with $d_{G_j}^{(m)}(v_j)\leq \delta^{(m)}(T_r(j))-1$,
then let $G_{j-1}:=G_j-\{v_j\}$ and it also follows that $\N(G_{j-1},K_m)=\N(G_j,K_m)-d_{G_j}^{(m)}(v_j)\geq \N(T_r(j-1),K_m)+(n-j+1)$;
otherwise, terminate.

Let $G_t$ be the graph for which the above iteration terminates. So $G_t$ has exactly $t$ vertices and $\delta^{m}(G_t)\ge \delta^{m}(T_r(t))$.
Suppose that $t<n_0$. Then we have
$$\binom{n_0}{m}>\binom{t}{m}\geq \N(G_t,K_m)\geq \N(T_r(t),K_m)+(n-t)\geq n-n_0\geq \binom{n_0}{m},$$
a contradiction. So we have $n\geq |V(G_t)|=t\geq n_0$.

Now suppose that under the additional condition $\delta^{m}(G_t)\ge \delta^{m}(T_r(t))$,
one can derive from the inequality $\N(G_t,K_m)\geq \N(T_r(t),K_m)$ (for $t\ge n_0$) that
\eqref{equ:N(G,Km)} holds for $G_t$, i.e., $\N(G_t,K_m)\le \N(T_r(t),K_m)+\biex(t,H)\cdot O(t^{m-2})$.
Then we would infer that \eqref{equ:N(G,Km)} also holds for $G$, by the following
\begin{eqnarray*}
\N(G,K_m)&=&\N(G_t,K_m)+\sum_{j=t+1}^{n}d_{G_j}^{(m)}(v_j)\\&\leq &\N(T_r(t),K_m)+\biex(t,H)\cdot O(t^{m-2})+\sum_{j=t+1}^{n}\delta^{(m)}(T_r(j))\\
&=&\N(T_r(n),K_m)+\biex(n,H)\cdot O(n^{m-2}),
\end{eqnarray*}
where the last equality follows from Proposition \ref{prop:turan2}. This proves Claim \ref{LEM: large deltam}.
\end{proof}

Choose $\varepsilon>0$ to be sufficiently small.
Let $V_1,\cdots,V_r$ be a partition of $V(G)$ such that $\sum_{i=1}^re(V_i)$ is minimized.
In view of \eqref{equa: N(G,Km)>N(Trn,Km)}, by Theorem \ref{THM: partition}, we have that
\begin{equation}
\sum_{i=1}^re(V_i)<\varepsilon n^2. \label{equa: sum e(V_i)<ep n^2}
\end{equation}
By Lemma \ref{LEM:InEdges},
there exists some $\gamma=\gamma(\varepsilon)$ with $\lim_{\varepsilon\rightarrow0}\gamma(\varepsilon)=0$ such that
\begin{equation}\label{equ:e[r]}
e(G[V_1,\cdots,V_r])>e(T_r(n))-\gamma n^2.
\end{equation}
Let $\beta=\beta(\varepsilon):=\max(2\sqrt{\varepsilon},\sqrt[3]{4\gamma})$.
We may assume that $\varepsilon$ is small so that $\beta<(r-1)^{-2}$. Let $B_i=\{x\in V_i:|N(v)\cap V_i|>\beta n\}$ for $i\in [r]$.
Let $B=\cup_{i=1}^rB_i$ and let $U_i=V_i\setminus B$.
Because of \eqref{equa: sum e(V_i)<ep n^2} and $\beta\geqs2\sqrt{\varepsilon}$, we get
$$|B|<\frac{2\varepsilon n^2}{\beta n}\leqs\frac{\beta }2n.$$
The next claim further bounds the size of $B$  from above by an absolute constant.
Recall the definition of $\sigma(H)$ in Section 2.

\begin{claim}\label{CLM: few bad vertices}
There exists some positive constant $K$ depending only on $\beta$ and $H$ such that $|B|\leqs K(\sigma(H)-1)$. In particular, if $H$ is edge-critical, then $B=\emptyset$.
\end{claim}
\begin{proof}
Note that $V_1,\cdots,V_r$ is a partition of $V(G)$ such that $\sum_{i=1}^re(V_i)$ is minimized.
So for any $v\in V_i$ and $j\neq i$, we have $|N(v)\cap V_j|\geq |N(v)\cap V_i|$.
This together with the definition of $B$ show that for any $v\in B$ and $i\in [r]$, $|N(v)\cap V_i|\geq \beta n$.
Since $U_i=V_i\setminus B$ and $|B|<\frac\beta2n$, it follows that $|N(v)\cap U_i|>\frac\beta2n$.
	
Consider an arbitrarily but fixed $v\in B$. Let $S_i\subset N(v)\cap U_i$ be a set of size $\frac\beta2n$ for each $i\in[r]$.
The inequality \eqref{equ:e[r]} tells that $G[V_1,\cdots,V_r])$ misses at most $\gamma n^2$ edges,
so for all $i\neq j$ we have $e(G[S_i,S_j])>|S_i||S_j|-\gamma n^2\geq (1-\beta)\beta^2n^2/4$.
(Here, we used $\beta\geq \sqrt[3]{4\gamma}$.) Thus the edge-density of $G[\cup_{i=1}^rS_i]$ is at least
$$\frac{\binom{r}{2}(1-\beta)\beta^2n^2/4}{\binom{r\beta n/2}{2}}\geqs\frac{r-1}{r}(1-\beta)>\frac{r-2}{r-1},$$
where the last inequality holds because of that $\beta<(r-1)^{-2}$.
So we can apply the supersaturation theorem of Erd\H{o}s and Simonovits \cite{Erdos-Simononvits-supersaturation}
and conclude that the graph $G[\cup_{i=1}^rS_i]$ contains at least $cn^{br}$ copies of the $b$-blowup $K_{r}(b)$,
where $b:=|V(H)|$ and $c:=c(\beta,H)>0$ is a constant.
	
Let $\mathcal X$ be the set of all copies of $K_r(b)$ in $G[\cup_{i=1}^rS_i]$.
So $|\mathcal X|\leq n^{br}$.
We then define an auxiliary bipartite graph $\mathcal{G}$ with the bipartition $(\mathcal X,B)$,
where $R\in \mathcal X$ and $v\in B$ are adjacent in $\mathcal{G}$ if and only if $V(R)\subseteq N_{G}(v)$.
By the previous paragraph, we see $d_\mathcal{G}(v)\geq cn^{br}$ for all $v\in B$.
We point out that $d_\mathcal{G}(R)\leq \sigma(H)-1$ for all $R\in\mathcal X$,
as otherwise it will lead to an $H$-copy by the definition of $\sigma(H)$.
Therefore, $|B|cn^{br}\leq e(\mathcal{G})\leq(\sigma(H)-1)\cdot n^{br}$.
This shows that $|B|\leq K(\sigma(H)-1)$, where $K=1/c$.
\end{proof}

\begin{claim}\label{CLM: close balanced}
There exists some $\theta=\theta(\varepsilon)$ with $\lim_{\varepsilon\rightarrow0}\theta(\varepsilon)=0$ such that $||V_i|-\frac{n}{r}|<\theta n$ for all $i\in[r]$.
\end{claim}
\begin{proof} By symmetry, it suffices for us to prove for $i=1$.
Let $p:=|V_1|/n\in[0,1]$.
Let $K_{V_1,\cdots,V_r}$ be the complete $r$-partite graph with parts $V_1,\cdots,V_r$.
Each $K_m$-copy in $G$ either contains some edge in $\cup_{i=1}^{r}E(G[V_i])$ or is contained in $G[V_1,\cdots,V_r]\subseteq K_{V_1,\cdots,V_r}$.
By \eqref{equa: N(G,Km)>N(Trn,Km)} and \eqref{equa: sum e(V_i)<ep n^2}, it follows that
\begin{equation*}
\N(T_r(n),K_m)\le \N(G,K_m)\leqs \varepsilon n^2\cdot n^{m-2}+\N(K_{V_1,\cdots,V_r},K_m)
\end{equation*}
Since $K_m$-copy in $K_{V_1,\cdots,V_r}$ either contains exactly one vertex in $V_1$ or is contained in $K_{V_2,\cdots,V_r}$, we have
\begin{equation*}
\N(K_{V_1,\cdots,V_r},K_m)\leq |V_1|\cdot\N(K_{V_2,\cdots,V_r},K_{m-1})+\N(K_{V_2,\cdots,V_r},K_m).
\end{equation*}
By Theorem \ref{THM: Erdos complete}, we also have that for $j\in \{m-1, m\}$
\begin{equation*}
\N(K_{V_2,\cdots,V_r},K_j)\leqs\N(T_{r-1}(n-|V_1|),K_j).
\end{equation*}
Putting the above inequalities together, it holds that $$\N(T_{r}(n),K_m)\leq \varepsilon n^m+|V_1|\cdot \N(T_{r-1}(n-|V_1|),K_{m-1})+\N(T_{r-1}(n-|V_1|),K_m).$$
By Proposition \ref{prop:turan1}, this yields
	\begin{eqnarray*}
	\binom rm\left(\frac nr\right)^m &\leqs& pn\binom{r-1}{m-1}\left(\frac{(1-p)n}{r-1}\right)^{m-1}+\binom{r-1}{m}\left(\frac{(1-p)n}{r-1}\right)^m+2\varepsilon n^m.
	\end{eqnarray*}
After some simplifications, it gives that
	$$f(p):=m(r-1)p(1-p)^{m-1}+(r-m)(1-p)^m-r\left(1-\frac1r\right)^m\geq-2\varepsilon.$$
One can easily verify that $f(p)$ increases in $[0,\frac1r]$ and decreases in $[\frac1r,1]$, where $f(\frac1r)=0$.
So by the continuity of $f$, there exists some $\theta=\theta(\varepsilon)$ with $\lim_{\varepsilon\rightarrow0}\theta(\varepsilon)=0$ such that $|p-\frac1r|<\theta$.
This proves Claim \ref{CLM: close balanced}.
\end{proof}

\begin{claim}\label{CLM: degree of good is large}
	There exists $\eta=\eta(\varepsilon)$ with $\lim_{\varepsilon\rightarrow0}\eta(\varepsilon)=0$ such that $$|N(v)\cap U_j|>\left(\frac1r-\eta\right) n$$
	for every $v\in U_i$ and every $j\neq i$.
\end{claim}
\begin{proof}
Fix a vertex $v\in U_i$ and some $j\neq i$.
We will show this claim by estimating $d^{(m)}_G(v)$.

First let us estimate the number of $K_m$-copies containing $v$ in $G[V_1,\cdots,V_r]$.
Such copies may contain some vertex in $V_j$ or not.
By Claim \ref{CLM: close balanced}, we have $|V_k|<\frac nr+\theta n$ for each $k\in [r]$.
So the number of such copies containing some vertex in $V_j$ is at most
$$|N(v)\cap V_j|\cdot \binom{r-2}{m-2}\left(\frac nr+\theta n\right)^{m-2},$$
and the number of such copies containing no vertex in $V_j$ is at most $$\binom{r-2}{m-1}\left(\frac nr+\theta n\right)^{m-1}.$$

For each copy of $K_m$ in $G$ that contains $v$, if it is not in $G[V_1,\cdots,V_r]$,
then it contains either some neighbor of $v$ in $V_i$ or an edge in $G[V_k]$ for some $k\neq i$.
The number of $K_m$-copies of the former kind is at most $|N(v)\cap V_i|\cdot n^{m-2}\leq\beta n^{m-1}$,
and in view of \eqref{equa: sum e(V_i)<ep n^2}, the number of the latter kind is at most $\varepsilon n^2\cdot n^{m-3}\leq \varepsilon n^{m-1}$.
This shows that
\begin{equation*}
d_G^{(m)}(v)\leqs(\beta+\varepsilon) n^{m-1}+|N(v)\cap V_j|\cdot \binom{r-2}{m-2}\left(\frac nr+\theta n\right)^{m-2}+\binom{r-2}{m-1}\left(\frac nr+\theta n\right)^{m-1}.
\end{equation*}
Also by Claim \ref{LEM: large deltam}, we may assume that
$$\delta^{(m)}(G)\geqs\delta^{(m)}(T_r(n))=\binom{r-1}{m-1}\left(\frac nr\right)^{m-1}+O(n^{m-2}).$$
Putting the above two inequalities together, we have
$$\binom{r-1}{m-1}\frac 1{r^{m-1}}\leqs \beta+2\varepsilon+\frac{|N(v)\cap V_j|}n\binom{r-2}{m-2}\left(\frac1r+\theta\right)^{m-2}+\binom{r-2}{m-1}\left(\frac1r+\theta\right)^{m-1}.$$
It then follows that there exists some $\xi=\xi(\beta,\varepsilon,\theta)$ with $\lim_{\beta,\varepsilon,\theta\rightarrow0}\xi(\beta,\varepsilon,\theta)=0$ such that $|N(v)\cap V_j|>(\frac1r-\xi) n$.
Finally, recall that $|B_j|\leq |B|<\frac\beta2n$ (or use Claim \ref{CLM: few bad vertices} instead).
Thus by letting $\eta(\varepsilon):=\xi(\beta,\varepsilon,\theta)+\frac{\beta}2,$ we get that $|N(v)\cap U_j|\geqs|N(v)\cap V_j|-|B_j|>(\frac1r-\eta)n,$ completing the proof of Claim \ref{CLM: degree of good is large}.
\end{proof}

\begin{claim}\label{CLM: few inter-edges}
	For every $i\in[r]$, $e(U_i)\leqs\biex(n,H).$ In particular, if $H$ is edge-critical, then $U_1,\cdots,U_r$ are all independent sets.
\end{claim}
\begin{proof}
Suppose for a contradiction that say, $e(U_1)>\biex(n,H)$.
Then $G[U_1]$ contains some $F\in\mathcal F_H$. Let $b=|V(H)|$.
We assert that we can find $X_2,\cdots,X_r$ with $X_i\subset U_i$ and $|X_i|=b$ such that $G[V(F),X_2,\cdots,X_r]$ is a complete $r$-partite graph.
If so, then clearly $G[V(F)\cup X_2\cup\cdots\cup X_r]$ contains a copy of $H$, a contradiction.
	
To do this, suppose inductively that for some $i\in\{1,...,r-1\}$, we have obtained $X_2,\cdots,X_i$ such that $G[V(F),X_2,\cdots,X_i]$ is complete $i$-partite.
(For $i=1$, we just view it as the set $V(F)$.)
Then the number of common neighbors of $L_i:=V(F)\cup X_2\cup\cdots\cup X_i$ in $U_{i+1}$ is at least
$$\left(\sum_{v\in L_i}|N(v)\cap U_{i+1}|\right)-(|L_i|-1)|U_{i+1}|>|L_i|(\frac1r-\eta)n-(|L_i|-1)(\frac1r+\theta)n$$
$$\geq \left(\frac1r-|L_i|(\eta+\theta)\right)\cdot n\geq\left(\frac1r-br(\eta+\theta)\right)\cdot n.$$
Here, the first inequality follows from Claim \ref{CLM: degree of good is large} and
the fact $|U_{i+1}|\leqs|V_{i+1}|<(\frac1r+\theta)n$ (by Claim \ref{CLM: close balanced}),
and the last inequality holds as $|L_i|\leq bi\leq br$.
Since $\eta$ and $\theta$ are sufficiently small and $n$ is sufficiently large,
we can find the desired set $X_{i+1}\subset U_{i+1}$ with $|X_{i+1}|=b$,
proving Claim \ref{CLM: few inter-edges}.
\end{proof}

We are ready to prove the upper bound \eqref{equ:N(G,Km)} of $\N(G,K_m)$.
It is clear that every copy of $K_m$ in $G$ either is contained in $G[U_1,\cdots,U_r]$, or contains some edge in $\cup_{i=1}^rE(G[U_i])$, or contains some vertex in $B$.
Since $G[U_1,\cdots,U_r]$ is $K_{r+1}$-free, by Theorem \ref{THM: Erdos complete}, we have
$$\N(G[U_1,\cdots,U_r],K_m)\leq \N(T_r(n),K_m).$$
Since $\sum_{i=1}^re(U_i)\leq r\cdot \biex(n,H)$ (by Claim \ref{CLM: few inter-edges}) and every edge can be contained in at most $n^{m-2}$ copies of $K_m$,
the number of copies of $K_m$ that contain some edge in $\cup_{i=1}^rE(G[U_i])$ is at most
$$r\cdot \biex(n,H)\cdot n^{m-2}=\biex(n,H)\cdot O(n^{m-2}).$$
Lastly, since each vertex can be contained at most $n^{m-1}$ copies of $K_m$, the number of copies of $K_m$ that contain some vertex in $B$ is at most
$$|B|\cdot n^{m-1}\leq K(\sigma(H)-1)\cdot n^{m-1}\leq \biex(n,H)\cdot O(n^{m-2}),$$
where the first inequality follows by Claim \ref{CLM: few bad vertices} and the last inequality holds because of Proposition \ref{Obs: sigma VS biex}.
Putting the above together, we obtain the desired upper bound
\begin{equation*}
\N(G,K_m)\leqs\N(T_r(n),K_m)+\biex(n,H)\cdot O(n^{m-2}).
\end{equation*}
The proof of Theorem \ref{THM: Main Theorem} is completed.

Now suppose $H$ is edge-critical. By Claim \ref{CLM: few bad vertices}, $B=\emptyset$ and so $V(G)=U_1\dot{\cup}\cdots\dot{\cup}U_r$.
By Claim \ref{CLM: few inter-edges}, we see that $U_1,\cdots,U_r$ are all independent sets, implying that $G$ is $r$-partite and thus $K_{r+1}$-free.
Hence by Theorem \ref{THM: Erdos complete}, it holds that $\N(G,K_m)\leq \N(T_r(n),K_m)$, with the equality holds if and only if $G=T_r(n)$.
This proves Corollary \ref{cor:edge-critical}.\qed

\section{Counting complete multipartite graphs}\label{sec:rpartite}
Throughout this section let $r\geq 2$ and $t\geq s$ be fixed integers.
Let $K^{(r)}_{s,t}$ denote the complete $r$-partite graph with one part of size $t$ and the other $r-1$ parts of size $s$.
It is easy to see that Theorem \ref{cor:mulpart} will follow from the coming result.

\begin{thm}\label{THM: MAIN ex(n,K_{s:(r-1);t})}
Let $r\geq 2$ and $t\geq s$ be positive integers. Then the following hold:
\begin{itemize}
\item [(a)] If $t<s+\frac12+\sqrt{2s+\frac14}$, then for sufficiently large $n$,
the unique $n$-vertex $K_{r+1}$-free graph which maximizes the number of copies of $K^{(r)}_{s,t}$ is the Tur\'an graph $T_r(n)$.
\item [(b)] If $t=s+\frac12+\sqrt{2s+\frac14}$, then  $\ex(n,K^{(r)}_{s,t},K_{r+1})=(1+o(1))\cdot \N(T_r(n),K^{(r)}_{s,t})$.
Moreover, in case of $r=2$, $\ex(n,K_{s,t},K_3)\ge\N(T_2(n),K_{s,t})+\Omega(n^{s+t-2})$.
\item [(c)] If $t>s+\frac12+\sqrt{rs+\frac14}$, then there exists a constant $c=c(r,s,t)>0$ such that $\ex(n,K^{(r)}_{s,t},K_{r+1})\geq(1+c)\cdot \N(T_r(n),K^{(r)}_{s,t})$.
\end{itemize}
\end{thm}

In this section we will prove Theorem \ref{THM: MAIN ex(n,K_{s:(r-1);t})}, by assuming Lemmas \ref{LEM: DeltaFr>0} and \ref{LEM: r-partite T_r(n) is best} (see below; their proofs will be postponed to the next section).
Before introducing the lemmas, we will need to give some notations.

\begin{Definition}\label{def:F(a,n)}
For integers $a\leq n$, let $G^{r}_{a,n}$ be the complete $r$-partite graph $G$ on $n$ vertices with parts $V_1, V_2, ...,V_r$
such that $G[V_2\cup\cdots\cup V_r]=T_{r-1}(a)$.
Let $F_{r,s,t}(a,n)$ be the number of copies of $K^{(r)}_{s,t}$ in $G^{r}_{a,n}$ each of which contains a fixed vertex in $V_1$.
\end{Definition}

\noindent Let $\lambda_{s,t}$ be $\frac12$ if $s=t$ and $1$ otherwise.
Then $F_{r,s,t}(a,n)$ can be expressed as
\begin{equation}\label{equ:F1}
\lambda_{s,t}\cdot\left[\binom{n-1-a}{s-1}\cdot \N(T_{r-1}(a),K^{(r-1)}_{s,t})+\binom{n-1-a}{t-1}\cdot \N(T_{r-1}(a),K^{(r-1)}_{s,s})\right].
\end{equation}
In case that $a=\lfloor\frac{r-1}{r}n\rfloor$, we see that $G^r_{a,n}=T_r(n)$ and $G^r_{a,n}\backslash\{v\}=T_r(n-1)$ for any $v\in V_1$. Hence we have
\begin{equation}\label{equa:T_r(n)-T_r(n-1)=F_rst}
\N(T_r(n),K^{(r)}_{s,t})-\N(T_r(n-1),K^{(r)}_{s,t})=F_{r,s,t}\left(\left\lfloor\frac{r-1}{r}n\right\rfloor,n\right).
\end{equation}

\begin{lem}\label{LEM: DeltaFr>0} 
\begin{itemize}
	\item [(i)] If $s\leqs t< s+\frac12+\sqrt{2s+\frac14}$, then the following holds for sufficiently large $n$.
If $F_{r,s,t}(\lfloor\frac{r-1}{r}n\rfloor,n)\le F_{r,s,t}(d,n)$, then $d\ge\lfloor\frac{r-1}{r}n\rfloor$.
	\item[(ii)] If $t=s+\frac12+\sqrt{2s+\frac14}$, then for any $\varepsilon>0$, there exists a real $\eta>0$ such that the following holds for sufficiently large $n$.
	If $F_{r,s,t}(\lfloor\frac{r-1}{r}n\rfloor,n)\le F_{r,s,t}(d,n)+\eta n^{(r-1)s+t-1}$, then $d\ge\frac{r-1}{r}n-\varepsilon n$.
\end{itemize}
\end{lem}


\begin{lem}\label{LEM: r-partite T_r(n) is best}
For $s\leq t\le s+\frac12+\sqrt{2s+\frac14}$ and sufficiently large $n$,
let $G$ be an $n$-vertex $r$-partite graph which maximizes the number of copies of $K^{(r)}_{s,t}$.
Then $G$ is a complete $r$-partite graph with each part of size $\frac nr+o(n)$.
Moreover, if $t<s+\frac12+\sqrt{2s+\frac14}$, then $G=T_r(n)$ is unique.
\end{lem}

Now we are in a position to prove Theorem \ref{THM: MAIN ex(n,K_{s:(r-1);t})}.

\begin{proof}[Proof of Theorem \ref{THM: MAIN ex(n,K_{s:(r-1);t})} (Assuming Lemmas \ref{LEM: DeltaFr>0} and \ref{LEM: r-partite T_r(n) is best}).]
We first prove the ``moreover" part of $(b)$ and the case $(c)$, by indicating that
some complete $r$-partite graphs have more copies of $K_{s,t}^{(r)}$ than the Tur\'an graphs $T_r(n)$.
For the ``moreover" part of $(b)$, we have $t=s+\frac12+\sqrt{2s+\frac14}$ and $r=2$.
By some tedious but straightforward calculations, one can show for $x=\Theta(\sqrt n)$ that
\begin{equation*}
\N(K_{\frac n2-x,\frac n2+x},K_{s,t})-\N(T_2(n),K_{s,t})=2st\left(\frac n2\right)^{s+t-3}x^2-\frac{2st}{3}\left(\frac n2\right)^{s+t-4}x^4+o(n^{s+t-2}).
\end{equation*}
By letting $x=\frac{\sqrt{3n}}{2}+o(\sqrt n)$, the desired inequality follows by
\begin{equation*}
\ex(n,K_{s,t},K_3)\geq \N(K_{\frac n2-x,\frac n2+x},K_{s,t})\geq \N(T_2(n),K_{s,t})+\left(\frac{3st}{2}+o(1)\right)\left(\frac n2\right)^{s+t-2}.
\end{equation*}
For the case $(c)$, let $t>s+\frac12+\sqrt{rs+\frac14}$ and consider $K_{x_1n,\cdots,x_rn}$, where $x_i=x\in (0,\frac1{r-1})$ for $i\in [r-1]$ and $x_r=1-(r-1)x$.
It is not hard to see that
\begin{eqnarray*}
	\N(K_{x_1n,\cdots,x_rn},K^{(r)}_{s,t})=\sum_{i=1}^r\binom{x_in}{t}\prod_{j\neq i}\binom{x_jn}{s}=\frac{e^{F(x)}+o(1)}{t!(s!)^{r-1}}n^{(r-1)s+t},
\end{eqnarray*}
where $F(x)=s\log [x^{r-1}-(r-1)x^r]+\log [(r-1)x^{t-s}+(1-(r-1)x)^{t-s}].$ In particular,
$$\N(T_r(n),K^{(r)}_{s,t})=\frac{e^{F(\frac1r)}+o(1)}{t!(s!)^{r-1}}n^{(r-1)s+t}.$$
Therefore to prove the case $(c)$, it suffices to show that $\frac1r$ is not a maximum point of $F(x)$ in the interval $(0,\frac1{r-1})$;
and further, it is enough to show $F''(\frac1r)>0$.
This indeed is the case, as by some routine calculations one can show that
\begin{equation*}
F''\left(\frac1r\right)=r^2(r-1)\cdot [(t-s)^2-t-s(r-1)]>0,
\end{equation*}
where the inequality holds by $r\geq 2$ and $t>s+\frac12+\sqrt{rs+\frac14}$.

In the rest of the proof we assume $s\leq t\leq s+\frac12+\sqrt{2s+\frac14}$.
We will apply induction on $r$ to prove the remaining statements of Theorem \ref{THM: MAIN ex(n,K_{s:(r-1);t})},
namely for sufficiently large $n$,
\begin{itemize}\label{equ:d}
\item[(a).] $T_r(n)$ uniquely attains the maximum $\ex(n,K^{(r)}_{s,t},K_{r+1})$ if $s\le t< s+\frac12+\sqrt{2s+\frac14}$;
\item[(b).] $\ex(n,K^{(r)}_{s,t},K_{r+1})=\N(T_r(n),K^{(r)}_{s,t})+o(n^{(r-1)s+t})$ if $t= s+\frac12+\sqrt{2s+\frac14}$.
\end{itemize}
For the case $r=1$, we view $K^{(r)}_{s,t}$ and $T_r(n)$ as graphs with empty edge set on $t$ vertices and $n$ vertices respectively,
and then items (a) and (b) holds trivially.
Now suppose that these two items hold for the case $r-1$.

Let $n$ be sufficiently large, $\varepsilon>0$ be sufficiently small, and $\eta$ be obtained from Lemma \ref{LEM: DeltaFr>0} (ii)
such that $\frac{r-1}r-\varepsilon>\frac{3r-4}{3r-1}$.
Let $G$ be an $n$-vertex $K_{r+1}$-free graph which maximizes the number of copies of $K^{(r)}_{s,t}$.
So we have
\begin{equation}\label{equ: N(G,Kst)}
\N(G,K^{(r)}_{s,t})=\ex(n,K^{(r)}_{s,t},K_{r+1})\geq \N(T_r(n),K^{(r)}_{s,t}).
\end{equation}

We then recursively define a sequence of graphs $G_i$'s as following.
Let $G_n:=G$. For $i\leq n$, if there is some vertex $v_i\in V(G_i)$ with $d_{G_i}(v_i,K^{(r)}_{s,t})\leq \delta_i-1$,
where $$\delta_i:=\N(T_r(i),K^{(r)}_{s,t})-\N(T_r(i-1),K^{(r)}_{s,t}), $$
then let $G_{i-1}=G_i\backslash \{v_i\}$ and continue; otherwise, terminate.
Suppose this recursive process stops at $H:=G_\ell$ for some $\ell\leq n$.
Then $H$ has $\ell$ vertices with $\delta(H,K^{(r)}_{s,t})\geq \delta_\ell$ and
\begin{align}
\N(H, K^{(r)}_{s,t})&=\N(G, K^{(r)}_{s,t})-\sum_{i=\ell+1}^n d_{G_i}(v_i,K^{(r)}_{s,t})\label{equ:N(H)}\\
&\geq\N(T_r(n),K^{(r)}_{s,t})-\sum_{i=\ell+1}^n \delta_i+(n-\ell)=\N(T_r(\ell),K^{(r)}_{s,t})+(n-\ell).
\end{align}
Assume that $n\geq n_0+n_0^{(r-1)s+t}$ for some sufficiently large $n_0$.
We claim that $\ell\geq n_0$;
as otherwise $n_0>\ell$, from which it follows that $$n_0^{(r-1)s+t}>\N(H,K^{(r)}_{s,t})\geq \N(T_r(\ell),K^{(r)}_{s,t})+(n-\ell)\geq n-n_0\geq n_0^{(r-1)s+t},$$ a contradiction.

Let $v\in V(H)$ have minimum degree $d_v$ in $H$.
We claim that $d_H(v,K^{(r)}_{s,t})$ is at most
\begin{equation}\label{equa: d<Frst}
\lambda_{s,t}\cdot \left[\binom{\ell-1-d_v}{s-1}\cdot \ex(d_v,K^{(r-1)}_{s,t},K_{r})+\binom{\ell-1-d_v}{t-1}\cdot \ex(d_v,K^{(r-1)}_{s,s},K_{r})\right].
\end{equation}
Note that as $H$ is $K_{r+1}$-free, $H[N_H(v)]$ is $K_r$-free.
Every $K^{(r)}_{s,t}$-copy $T$ in $H$ containing $v$ must contain either $(r-2)s+t$ vertices in $N_H(v)$ which induce a copy of $K^{(r-1)}_{s,t}$,
or $(r-1)s$ vertices in $N_H(v)$ which induce a copy of $K^{(r-1)}_{s,s}$.
Moreover, if the former case occurs, then the other $s-1$ vertices of $T$ must be in $V(H)\backslash (N_H(v)\cup \{v\})$,
as otherwise it will lead to a copy of $K_r$ in $H[N_H(v)]$;
similarly, if the later one occurs, then the other $t-1$ vertices of $T$ must be in $V(H)\backslash (N_H(v)\cup\{v\})$.
This justifies the claim.

Let $\mu=0$ if $s\leq t<s+\frac12+\sqrt{2s+\frac14}$, and $\mu=1$ otherwise.
By \eqref{equa:T_r(n)-T_r(n-1)=F_rst}, we have
\begin{equation}\label{equ:d_H}
d_H(v,K^{(r)}_{s,t})\geq \delta(H,K^{(r)}_{s,t})\geq \delta_\ell=F_{r,s,t}\left(\left\lfloor\frac{r-1}{r}\ell\right\rfloor,\ell\right)\geq \Omega(\ell^{(r-1)s+t-1}).
\end{equation}
Then by \eqref{equa: d<Frst}, $\ell^{s-1}d_v^{(r-2)s+t}+\ell^{t-1}d_v^{(r-1)s}\geq d_H(v,K^{(r)}_{s,t})\geq \Omega(\ell^{(r-1)s+t-1}),$
which implies that $d_v=\Omega(\ell)=\Omega(n_0)$ is sufficiently large.
By our induction, it follows that
\begin{equation*}
\ex(d_v,K^{(r-1)}_{s,t},K_{r})=\N(T_{r-1}(d_v),K^{(r-1)}_{s,t})+\mu\cdot o(d_v^{(r-2)s+t-1}).
\end{equation*}
This, together with \eqref{equa: d<Frst} and \eqref{equ:F1} (i.e., the definition  of $F_{r,s,t}$), implies that
\begin{equation*}\label{equa: d<Frst2}
d_H(v,K^{(r)}_{s,t})\leqs F_{r,s,t}(d_v,\ell)+\mu\cdot o(\ell^{(r-1)s+t-1}).
\end{equation*}
By \eqref{equ:d_H}, for sufficiently large $\ell$ (as $\ell\geq n_0$), we have
\begin{equation*}
F_{r,s,t}\left(\left\lfloor\frac{r-1}{r}\ell\right\rfloor,\ell\right)\leqs F_{r,s,t}(d_v,\ell)+\mu\cdot \eta\cdot\ell^{(r-1)s+t-1},
\end{equation*}
where $\eta$ is obtained from Lemma \ref{LEM: DeltaFr>0} (ii).
Applying Lemma \ref{LEM: DeltaFr>0}, we obtain that
the minimum degree $\delta(H)=d_v\geq (\frac{r-1}{r}-\varepsilon) \ell> \frac{3r-4}{3r-1}\ell.$
As $H$ is an $\ell$-vertex $K_{r+1}$-free graph,
by Theorem \ref{LEM: Andrasfai-Erdos-Sos} we see that $H$ is $r$-partite.
Then Lemma \ref{LEM: r-partite T_r(n) is best} shows that
\begin{equation*}
\N(H,K^{(r)}_{s,t})\leqs \N(T_r(\ell),K^{(r)}_{s,t})+\mu\cdot o(\ell^{(r-1)s+t}),
\end{equation*}
where the equality holds for $\mu=0$ if and only if $H=T_r(\ell)$.
By \eqref{equ: N(G,Kst)} and \eqref{equ:N(H)}, we have
\begin{align*}
\N(T_r(n),K^{(r)}_{s,t})&\leq \N(G, K^{(r)}_{s,t})=\N(H, K^{(r)}_{s,t})+\sum_{i=\ell+1}^n d_{G_i}(v_i,K^{(r)}_{s,t})\\
&\leq \N(T_r(\ell),K^{(r)}_{s,t})+\sum_{i=\ell+1}^n \delta_i+\mu\cdot o(\ell^{(r-1)s+t})-(n-\ell)\\
&= \N(T_r(n),K^{(r)}_{s,t}) +\mu\cdot o(\ell^{(r-1)s+t})-(n-\ell).
\end{align*}
If $s\leq t<s+\frac12+\sqrt{2s+\frac14}$ (that is, $\mu=0$),
then it is easy to see that $n=\ell$, $G=H$ and $\N(H, K^{(r)}_{s,t})=\N(T_r(n),K^{(r)}_{s,t})$;
and in this case Lemma \ref{LEM: r-partite T_r(n) is best} also shows that $G=H=T_r(n)$ is unique.
For the case $t=s+\frac12+\sqrt{2s+\frac14}$, it is also easy to see that $\N(G, K^{(r)}_{s,t})=\N(T_r(n),K^{(r)}_{s,t})+o(n^{(r-1)s+t})$.
The proof of Theorem \ref{THM: MAIN ex(n,K_{s:(r-1);t})} is completed.
\end{proof}

\section{Two Lemmas}
Here we prove Lemmas \ref{LEM: DeltaFr>0} and \ref{LEM: r-partite T_r(n) is best}.
Throughout this section, let $r, s, t$ be fixed integers such that $r\geq 2$ and $s\leq t\leq s+\frac12+\sqrt{2s+\frac14}$, and let $n$ be sufficiently large.

\subsection{Proof of Lemma \ref{LEM: DeltaFr>0}}
Recall the definition of $\lambda_{s,t}$, and let $\tilde\lambda_{s,t,r}=r-1$ if $t\neq s$ and $1$ otherwise.
One can easily obtain the following.
\begin{prop}\label{FACT: Asym of N(Tr(n),Kst)}
	$\N(T_r(n),K^{(r)}_{s,t})=(1+o(1))\frac{\tilde\lambda_{s,t,r+1}}{(s!)^{r-1}t!}(\frac{n}{r})^{(r-1)s+t}$.
\end{prop}

\begin{prop}\label{FACT: Asym of F_{r,s,t}}
$F_{r,s,t}\left(\left\lfloor\frac{r-1}{r}n\right\rfloor,n\right)=(1+o(1))\frac{\lambda_{s,t}(s\tilde{\lambda}_{s,t,r}+t)}{(s!)^{r-1}t!}\left(\frac{n}{r}\right)^{(r-1)s+t-1}.$
\end{prop}

From now on we will often write $F(a)$ instead of $F_{r,s,t}(a,n)$ for short.

\begin{prop}\label{LEM: d>Omega(n)}
There exist $\eta_0>0$ and $\gamma>0$ such that the following holds. For any $\eta\in [0,\eta_0)$, if $F(\lfloor\frac{r-1}{r}n\rfloor)\le F(d)+\eta n^{(r-1)s+t-1}$, then $d\ge\gamma n$.
\end{prop}
\begin{proof}
By Proposition \ref{FACT: Asym of F_{r,s,t}}, there is some $c>0$ such that
$F\left(\left\lfloor\frac{r-1}{r}n\right\rfloor\right)>cn^{(r-1)s+t-1}$. Let $\eta_0:=\frac{c}2$.
Suppose $0\le\eta\le\eta_0$ and $F(\lfloor\frac{r-1}{r}n\rfloor)\le F(d)+\eta n^{(r-1)s+t-1}$.
Then by the definition of $F$, we have that
$$n^{s-1}d^{(r-2)s+t}+n^{t-1}d^{(r-1)s}\geq F(d)\ge F\left(\left\lfloor\frac{r-1}{r}n\right\rfloor\right)-\eta n^{(r-1)s+t-1}\ge\frac{c}2n^{(r-1)s+t-1}.$$
This yields some $\gamma=\gamma(r,s,t)>0$ such that $d\ge\gamma n$.
\end{proof}

The following two propositions assert some properties on $F(a)$. We leave the technical details of their proofs in the Appendix \ref{app:A}.

\begin{prop}\label{LEM: Frst(a+1)>Frst(a) when a<1-1/r-eps}
For any $\gamma,\varepsilon>0$ with $\gamma+\varepsilon<\frac{r-1}{r}$, the following hold.
	\begin{itemize}
		\item [(i)] If $t< s+\frac12+\sqrt{2s+\frac14}$, then $F(a+1)>F(a)$ for all integers $a\in[\gamma n ,\lfloor\frac{r-1}{r}n\rfloor]$.
		\item [(ii)] If $t=s+\frac12+\sqrt{2s+\frac14}$, then $F(a+1)>F(a)$ for all integers $a\in[\gamma n ,(\frac{r-1}{r}-\varepsilon)n]$.
	\end{itemize}
\end{prop}

\begin{prop}\label{LEM: Frst d/n<r-1/r}
For any $\varepsilon\in(0,\frac{r-1}{r})$, there exists $\xi=\xi(\varepsilon,r,s,t)>0$ such that $F(\lfloor\frac {r-1}{r}n\rfloor)- F(\lfloor(\frac{r-1}{r}-\varepsilon) n\rfloor)>\xi n^{(r-1)s+t-1}$.
\end{prop}

We have collected all propositions needed for the proof of Lemma \ref{LEM: DeltaFr>0}.

\begin{proof}[Proof of Lemma \ref{LEM: DeltaFr>0}.]
First we consider the case (i) that $s\leq t<s+\frac12+\sqrt{2s+\frac14}$.
Suppose that $F\left(\left\lfloor\frac{r-1}{r}n\right\rfloor\right)\le F(d)$ (and $n$ is assumed to be sufficiently large throughout this section).
By Proposition \ref{LEM: d>Omega(n)}, there exists some $\gamma>0$ such that $d\ge\gamma n$. We may assume $\gamma<\frac{r-1}{r}$, as otherwise we are done.
Then by Proposition \ref{LEM: Frst(a+1)>Frst(a) when a<1-1/r-eps} (i), $F\left(\left\lfloor\frac{r-1}{r}n\right\rfloor\right)$ is the unique maximum of $F(a)$ in $[\gamma n,\lfloor\frac{r-1}{r}n\rfloor]$.
This yields that $d\ge\lfloor\frac{r-1}{r}n\rfloor$.

Now we consider the case (ii) that $t=s+\frac12+\sqrt{2s+\frac14}$.
For any $\varepsilon>0$, let $\eta_0$ and $\xi$ be obtained from Propositions \ref{LEM: d>Omega(n)} and \ref{LEM: Frst d/n<r-1/r} respectively.
Let $\eta:=\min\{\eta_0,\xi\}>0$ and write $v=(r-1)s+t-1$. Now suppose that $F\left(\left\lfloor\frac{r-1}{r}n\right\rfloor\right)\le F(d)+\eta n^{v}$.
Our goal is to show $d\ge\frac{r-1}{r}n-\varepsilon n$.

Suppose to the contrary that $d<\frac{r-1}{r}n-\varepsilon n$. By Lemma \ref{LEM: d>Omega(n)}, there exists some $\gamma>0$ such that $d\ge\gamma n$.
So $\gamma n\le d<(\frac{r-1}{r}-\varepsilon)n$.
Putting Proposition \ref{LEM: Frst(a+1)>Frst(a) when a<1-1/r-eps} (ii) and Proposition \ref{LEM: Frst d/n<r-1/r} together,
we have $F(d)\le F\left(\left\lfloor(\frac{r-1}{r}-\varepsilon)n\right\rfloor\right)<F\left(\left\lfloor\frac{r-1}{r}n\right\rfloor\right)-\xi n^{v}\leq F\left(\left\lfloor\frac{r-1}{r}n\right\rfloor\right)-\eta n^{v}$,
which is a contradiction to the assumption. This proves Lemma \ref{LEM: DeltaFr>0}.
\end{proof}

\subsection{Proof of Lemma \ref{LEM: r-partite T_r(n) is best}}
Let $G$ be an $n$-vertex $r$-partite graph with the maximum number of $K^{(r)}_{s,t}$-copies.
It is clear that $G$ must be a complete $r$-partite graph. So we may assume that $G=K_{a_1,\cdots,a_r}$ with $n=a_1+...+a_r$ and $a_r\geq\cdots\geq a_1\geq s$
(where $a_1\geq s$ is because $\N(G, \K)\geq 1$).

For any vector $\vec{x}=(x_1,...,x_r)$ with positive integers $x_i$'s, write $K_{\vec{x}}=K_{x_1,\cdots,x_r}$ and let
$$g(\vec{x})=\sum_{i=1}^r\binom{x_i}{t}\prod_{j\neq i}\binom{x_j}{s}~,~~*\vec{x}=(x_1+1,x_2,...,x_{r-1},x_r-1)~,\text{ and } \Delta g(\vec{x})=g(*\vec{x})-g(\vec{x})~.$$
Therefore, if $t\neq s$, then $g(\vec{x})=\N(K_{\vec{x}},K^{(r)}_{s,t})$; otherwise, $g(\vec{x})=r\N(K_{\vec{x}},K^{(r)}_{s,t})$.

We present a sequence of propositions as following.
\begin{prop}\label{prop:g(va)}
Let $\va=(a_1,...,a_r)$. Then we have $\Delta g(\va)\leq 0$.
\end{prop}
\begin{proof}
This clearly follows by the maximality of $\N(G,\K)$.
\end{proof}


\begin{prop}\label{CLM: a_1> gamma n}
There exists some $\gamma>0$ such that $a_1\ge\gamma n$.
\end{prop}
\begin{proof}
We have $\N(T_r(n),\K)\leq \N(G,\K)\leq \N(K_{a_1,n,...,n},\K)$.
Thus there exists some $c>0$ such that
$cn^{(r-1)s+t}\leq \N(G,\K)\leq (r-1)a_1^{s}n^{(r-2)s+t}+a_1^{t}n^{(r-1)s}.$
This implies that $a_1\ge\gamma n$ for some constant $\gamma>0$.
\end{proof}

For a vector $\vec{x}=(x_1,...,x_r)$, let $h(\vec{x})=x_rt!/s!\prod_{i=1}^r\binom{x_i}{s}$.

\begin{prop}\label{CLM: h(a)Dg(a)=}
Let $q=t-s$. The product $h(\va)\Delta g(\va)$ is equal to
\begin{equation*}
\frac{sa_r-t(a_1+1)}{a_1+1-s}(a_r-s)_{q}+\frac{ta_r-s(a_1+1)}{a_1+1-t}(a_1-s)_{q}+\frac{s(a_r-a_1-1)}{a_1+1-s}\sum_{i=2}^{r-1}(a_i-s)_{q}.
\end{equation*}		
\end{prop}
\begin{proof}
The proof is straightforward and we just give some computations here.
By routine calculations, we have $\Delta g(\va)=g(*\va)-g(\va)=A\cdot \prod_{i=2}^{r-1}\binom{a_i}{s}+ B\cdot \prod_{j\neq1,i,r}\binom{a_j}{s}$, where
$$A=\binom{a_1+1}{s}\binom{a_r-1}{t}+\binom{a_1+1}{t}\binom{a_r-1}{s}-\binom{a_1}{s}\binom{a_r}{t}-\binom{a_1}{t}\binom{a_r}{s},$$
$$B=\binom{a_1+1}{s}\binom{a_r-1}{s}-\binom{a_1}{s}\binom{a_r}{s}.$$
Using the formula $\binom{a_1+1}{x}\binom{a_r-1}{y}=\frac{a_1+1}{a_1+1-x}\cdot\frac{a_r-y}{a_r}\cdot \binom{a_1}{x}\binom{a_r}{y}$, one can derive that
$$\frac{\Delta g(\va)}{\prod_{i=1}^{r}\binom{a_i}{s}}=\frac{sa_r-t(a_1+1)}{(a_1+1-s)a_r}\frac{\binom{a_r}{t}}{\binom{a_r}{s}}+\frac{ta_r-s(a_1+1)}{(a_1+1-t)a_r}\frac{\binom{a_1}{t}}{\binom{a_1}{s}}
+\frac{s(a_r-a_1-1)}{(a_1+1-s)a_r}\sum_{i=2}^{r-1}\frac{\binom{a_i}{t}}{\binom{a_i}{s}}.$$
Now it follows easily by $h(\va)=\frac{a_rt!}{s!\prod_{i=1}^r\binom{a_i}{s}}$ and the formula $\binom{a_i}{t}=\frac{s!}{t!}\binom{a_i}{s}(a_i-s)_q$.
\end{proof}

For reals $x>0,\alpha\geq 0$ and an integer $k\geq 1$, let $(x)_k=\prod_{i=0}^{k-1}(x-i)$ and $H(x,\alpha)=H_1(x,\alpha)+H_2(x,\alpha)+H_3(x,\alpha)$, where
\begin{equation*}
\left\{
\begin{aligned}
	&H_1(x,\alpha)&=&~\left(s\alpha-q-\frac{qs+t}{x}\right)\left(1+\frac{1-q}{x}\right)\frac{(x+\alpha x)_q}{x^{q}},\\
	&H_2(x,\alpha)&=&~\left(t\alpha+q+\frac{qs-s}x\right)\left(1+\frac1x\right)\frac{(x)_q}{x^{q}},\\
	&H_3(x,\alpha)&=&~~(r-2)s\left(\alpha-\frac1x\right)\left(1+\frac{1-q}x\right)\frac{(x)_q}{x^{q}}.
\end{aligned}
\right.
\end{equation*}

\begin{prop}\label{CLM: H(x,a)<0}
	Let $\hat{x}=a_1-s$ and $\hat\alpha=\frac{a_r-a_1}{a_1-s}$. If $a_r\ge a_1+1$, then $H(\hat{x},\hat\alpha)\le0$.
\end{prop}
\begin{proof}	
Assume that $a_r\ge a_1+1$. Let $p(x)=\frac{x^{q+2}}{(x+1)(x+1-q)}$. We first show that
\begin{equation}
h(\va)\Delta g(\va)\ge H(\hat{x},\hat\alpha)\cdot p(\hat{x}).\label{equa: h(a)Dg(a)>Hp}
\end{equation}
One can rewrite the first two terms of $h(\va)\Delta g(\va)$ in Proposition \ref{CLM: h(a)Dg(a)=} as the following
\begin{align*}
&\frac{sa_r-t(a_1+1)}{a_1+1-s}(a_r-s)_{q}=\frac{(s\hat\alpha\hat{x}-q\hat{x}-qs-t)(\hat{x}+\hat\alpha\hat{x})_q}{\hat{x}+1}=H_1(\hat{x},\hat\alpha)\cdot p(\hat{x}),\\
&\frac{ta_r-s(a_1+1)}{a_1+1-t}(a_1-s)_{q}=\frac{(t\hat\alpha \hat{x}+q\hat{x}+qs-s)(\hat{x})_q}{\hat{x}+1-q}=H_2(\hat{x},\hat\alpha)\cdot p(\hat{x}).
\end{align*}
Thus to prove (\ref{equa: h(a)Dg(a)>Hp}), it suffices to show that the third term of $h(\va)\Delta g(\va)$ in Proposition \ref{CLM: h(a)Dg(a)=} is at least $H_3(\hat{x},\hat\alpha)p(\hat{x})$.
Indeed, since $\frac{s(a_r-a_1-1)}{a_1+1-s}\ge0$ and $\sum_{i=2}^{r-1}(a_i-s)_{q}\ge(r-2)(a_1-s)_q=(r-2)(\hat{x})_q$,
this follows by
$\frac{s(a_r-a_1-1)}{a_1+1-s}\sum_{i=2}^{r-1}(a_i-s)_{q}\ge\frac{(r-2)s(\hat\alpha \hat{x}-1)(\hat{x})_q}{\hat{x}+1}=H_3(\hat{x},\hat\alpha)\cdot p(\hat{x}).$

Next we use (\ref{equa: h(a)Dg(a)>Hp}) to show $H(\hat{x},\hat\alpha)\le0$.
Since $n$ is sufficiently large and $a_1\ge\gamma n$ (by Proposition \ref{CLM: a_1> gamma n}), it holds that $p(\hat{x})=p(a_1-s)>0$ and $h(\va)>0$;
also by Proposition \ref{prop:g(va)}, we have $\Delta g(\va)\le0$.
Therefore one can easily derive from \eqref{equa: h(a)Dg(a)>Hp} that $H(\hat{x},\hat\alpha)\le0$.
\end{proof}	

We also need the following properties on $H(x,\alpha)$, whose technical proofs can be found in Appendix \ref{app:C}.

\begin{prop}\label{LEM: Key tech lem to pv complete case}
$(i)$ For any fixed $C>\varepsilon>0$, there exists $x_0$ such that the following holds. If $x\ge x_0$ and $C\ge \alpha\ge\varepsilon $, then $H(x,\alpha)>0$.\\
$(ii)$ If $t< s+\frac12+\sqrt{2s+\frac14}$, then there exist $\varepsilon_0$ and $x_1$ such that the following holds. If $x\ge x_1$ and $\varepsilon_0\ge \alpha\ge\frac2x$, then $H(x,\alpha)>0$.
\end{prop}

Now we can finish the proof of Lemma \ref{LEM: r-partite T_r(n) is best}.

\begin{proof}[Proof of Lemma \ref{LEM: r-partite T_r(n) is best}.]
By Proposition \ref{CLM: a_1> gamma n}, there exists some $\gamma>0$ such that $a_1\ge\gamma n$.
Let $n$ be sufficiently large, $\hat x=a_1-s$, and $\hat\alpha=\frac{a_r-a_1}{a_1-s}$.

First we prove that $a_r-a_1=o(n)$, which would imply that $a_i=n/r+o(n)$.
Suppose to the contrary that $a_r-a_1\ge\varepsilon n$ for some $\varepsilon>0$.
As $n$ is sufficiently large, it follows that $2/\gamma\ge\hat\alpha\ge\varepsilon$.
Let $x_0$ be obtained from Proposition \ref{LEM: Key tech lem to pv complete case} (i) by applying with $C=2/\gamma$ and $\varepsilon$.
Since $\hat x=a_1-s\ge\gamma n-s \ge x_0$, by Proposition \ref{LEM: Key tech lem to pv complete case} (i) we get $H(\hat x,\hat\alpha)>0$,
which contradicts Proposition \ref{CLM: H(x,a)<0}.

Next we assume $t< s+\frac12+\sqrt{2s+\frac14}$ and aim to show that $G=T_r(n)$, or equivalently $a_r-a_1\leq 1$.
Assume that $a_r-a_1\geq 2$. Let $\varepsilon_0$ and $x_1$ be obtained from Proposition \ref{LEM: Key tech lem to pv complete case} (ii).
As we just prove $a_r-a_1=o(n)$, for sufficiently large $n$ we have $a_r-a_1\le \frac{\gamma\varepsilon_0}{2}n$.
This implies that $\varepsilon_0\ge\hat\alpha\ge\frac2{\hat x}$.
Also we have $\hat x\geq\gamma n-s\ge x_1$,
so by Proposition \ref{LEM: Key tech lem to pv complete case} (ii), we obtain $H(\hat x,\hat\alpha)>0$,
again a contradiction to Proposition \ref{CLM: H(x,a)<0}.
Now the proof of Lemma \ref{LEM: r-partite T_r(n) is best} is completed.
\end{proof}

\section{Concluding remarks}
In this paper we consider the generalized Tur\'an numbers $\ex(n,T,H)$ for graphs $T, H$ with $\chi(T)<\chi(H)$.
In the case that $T$ is a clique, Theorem \ref{THM: Main Theorem} gives a sharp estimate.
A natural question will be to consider for non-clique $T$.
Theorem \ref{THM: MAIN ex(n,K_{s:(r-1);t})} provides some answers for complete multipartite graphs $T$.
However, even for this case there lacks of evidences to speculate extremal graphs in general.
A special problem which we encounter with is that if, for $(T,H)=(K_{s,t},K_3)$ and $t\geq s+\frac12+\sqrt{2s+\frac14}$,
the extremal graphs are always bipartite.
If this is the case then one may expect to solve the problem similar as in Lemma \ref{LEM: r-partite T_r(n) is best}.
It also seems plausible to ask the extremal graphs for $\ex(n,T,K_r)$ for edge-critical graphs $T$ (in particular, for $\ex(n,C_{2k+1},K_r)$ where $r\geq 4$).
Our attempt to generalize Theorem \ref{THM: MAIN ex(n,K_{s:(r-1);t})} is limited by our capability of computation,
therefore it will be interesting to see if there exists some novel approach which can work for general problems.

\appendix

\section{Proofs of Propositions \ref{LEM: Frst(a+1)>Frst(a) when a<1-1/r-eps} and \ref{LEM: Frst d/n<r-1/r}}\label{app:A}
We begin by defining some functions: (let $q=t-s$ and $1/C=(s!)^{r-1}t!(r-1)^{(r-2)s+t-1}$)
\begin{equation*}
\Delta(a)=(F(a+1)-F(a))/\lambda_{s,t}~, ~~  M(a)=C{a^{(r-2)s+t}(n-a)^{s-2}}~,   \text{ and }
\end{equation*}	
\begin{equation*}
H(z)=s\lambda_{s,t}(s\tilde{\lambda}_{s,t,r-1}+t)z-(s^2-s)\tilde\lambda_{s,t,r}/(r-1)+st(r-1)^qz^{q+1}
-(t^2-t)(r-1)^{q-1}z^{q}.\label{equa: def of H(z)}
\end{equation*}

First we will need to prove the following two claims.
\begin{claim}\label{LEM: D=MH}
For $\gamma n\le a\le\left(1-\varepsilon\right)n$, it holds for sufficiently large $n$ that $\Delta(a)=M(a)\cdot \left[H\left(\frac{n-a}{a}\right)+o(1)\right],$ where $o(1)$ tends to $0$ as $n$ goes to infinity.
\end{claim}
\begin{proof}
We need to compute $\Delta(a)$. Write $N_1(a)=\N(T_{r-1}(a),K^{(r-1)}_{s,t})$ and $N_2(a)=\N(T_{r-1}(a),K^{(r-1)}_{s,s})$.
By the definition of the function $F$, we have
\begin{eqnarray} \label{equa: Frst(a,n)}
\frac{F(a)}{\lambda_{s,t}}=\binom{n-a-1}{s-1}N_1(a)+\binom{n-a-1}{t-1}N_2(a),
\end{eqnarray}
By \eqref{equa:T_r(n)-T_r(n-1)=F_rst},
$N_1(a+1)=N_1(a)+\delta_1$ and $N_2(a+1)=N_2(a)+\delta_2$,
where $\delta_1=F_{r-1,s,t}(\lfloor\frac{r-2}{r-1}(a+1)\rfloor,a+1)$ and $\delta_2=F_{r-1,s,s}(\lfloor\frac{r-2}{r-1}(a+1)\rfloor,a+1).$
So we can obtain that
\begin{eqnarray}
\frac{F(a+1)}{\lambda_{s,t}}&=&\binom{n-a-2}{s-1}(N_1(a)+\delta_1) +\binom{n-a-2}{t-1}(N_2(a)+\delta_2).\label{equa: Frest(a+1,n)}
\end{eqnarray}
By \eqref{equa: Frst(a,n)} and \eqref{equa: Frest(a+1,n)}, it follows that
\begin{equation*}
\Delta(a)=\binom{n-a-2}{s-1}\delta_1-\binom{n-a-2}{s-2}N_1(a)+\binom{n-a-2}{t-1}\delta_2-\binom{n-a-2}{t-2}N_2(a).\label{equa: Delta(a,n)}
\end{equation*}
Applying Propositions \ref{FACT: Asym of N(Tr(n),Kst)} and \ref{FACT: Asym of F_{r,s,t}} to $N_1(a),N_2(a),\delta_1$ and $\delta_2$,
one can derive that
\begin{eqnarray*}
		\Delta(a)&=(n-a)^{s-1}\left[\frac{\lambda_{s,t}(s\tilde{\lambda}_{s,t,r-1}+t)}{(s-1)!(s!)^{r-2}t!}+o(1)\right]\left(\frac{a}{r-1}\right)^{(r-2)s+t-1}\\
		&-(n-a)^{s-2}\left[\frac{\tilde\lambda_{s,t,r}}{(s-2)!(s!)^{r-2}t!}+o(1)\right]\left(\frac{a}{r-1}\right)^{(r-2)s+t}\\
		&+(n-a)^{t-1}\left[\frac{s}{(t-1)!(s!)^{r-1}}+o(1)\right]\left(\frac{a}{r-1}\right)^{(r-1)s-1}\\
		&-(n-a)^{t-2}\left[\frac{1}{(t-2)!(s!)^{r-1}}+o(1)\right]\left(\frac{a}{r-1}\right)^{(r-1)s}.
\end{eqnarray*}
Let $z=\frac{n-a}{a}$. After some simplifications, one can obtain that
\begin{eqnarray*}
		\Delta(a)&=Ca^{(r-1)s+t-2}\cdot\big[s\lambda_{s,t}(s\tilde{\lambda}_{s,t,r-1}+t)z^{s-1}-\frac{s(s-1)\tilde\lambda_{s,t,r}}{r-1}z^{s-2}+st(r-1)^qz^{t-1}\\
		&-(t^2-t)(r-1)^{q-1}z^{t-2}+o(1)\big]=M(a)\cdot [H(z)+o(1)].
\end{eqnarray*}
This proves Claim \ref{LEM: D=MH}.
\end{proof}

\begin{claim} \label{LEM: H is increasing}
$H(z)$ is strictly increasing in $[\frac1{r-1},+\infty)$ and $H(\frac1{r-1})\ge0$. Moreover, $H(\frac1{r-1})=0$ if and only if $r=2$ and $t=s+\frac12+\sqrt{2s+\frac14}$.
\end{claim}
\begin{proof}	
If $t=s$, then $H(z)=s\cdot\frac12(s+s)z-\frac{s^2-s}{r-1}+s^2z-\frac{s^2-s}{r-1}={2s^2}\left(z-\frac{1}{r-1}+\frac{1}{s(r-1)}\right).$
It is obvious that $H(\frac1{r-1})>0$ and $H(z)$ is strictly increasing.
	
Next we consider $s<t\le s+\frac12+\sqrt{2s+\frac14}$. Then $q\geq 1$ and $2s+q-q^2\geq 0$, where $2s+q-q^2=0$ if and only if $t= s+\frac12+\sqrt{2s+\frac14}$.
In this case we have
\begin{equation}\label{equa: H(z)}
H(z)=(s^2(r-2)+st)z-(s^2-s)+st(r-1)^qz^{q+1}-(t^2-t)(r-1)^{q-1}z^q.
\end{equation}
This implies that $H\left(\frac{1}{r-1}\right)=\frac{sr+q-q^2}{r-1}\geq 0,$
where the equality holds if and only if $r=2$ and $t= s+\frac12+\sqrt{2s+\frac14}$.
It remains to show $H(z)$ is strictly increasing in $[\frac1{r-1},+\infty)$.
To do so, it suffices to prove $H'(z)$ is strictly increasing in $[\frac1{r-1},+\infty)$ and $H'(\frac1{r-1})\geq 0$.

By (\ref{equa: H(z)}), one can obtain
\begin{eqnarray}
	&&H'(z)=s^2(r-2)+st+st(r-1)^q(q+1)z^q-(t^2-t)(r-1)^{q-1}qz^{q-1},\label{equa: H'}\\
	&&H''(z)=tq(r-1)^{q-1}z^{q-2}[s(r-1)(q+1)z-(t-1)(q-1)].\label{equa: H''}
\end{eqnarray}
So for $z>0$, $H''(z)\geq 0$ is equivalent to that $h(z):=s(r-1)(q+1)z-(t-1)(q-1)\geq 0$.
Since $h(z)$ is strictly increasing and $h\left(\frac1{r-1}\right)=(2s+q-q^2)+(q-1)\geq 0$,
we infer that $h(z)>0$ for $z>\frac1{r-1}$.
This also yields that $H''(z)>0$ for $z>\frac1{r-1}$. Therefore $H'(z)$ is strictly increasing in $[\frac1{r-1},+\infty)$.
Lastly, it follows from (\ref{equa: H'}) that
$H'\left(\frac{1}{r-1}\right)=(r-2)s^2+t(2s+q-q^2)\ge0.$
Now the proof of Claim \ref{LEM: H is increasing} is completed.	
\end{proof}		

We are ready to prove Propositions \ref{LEM: Frst(a+1)>Frst(a) when a<1-1/r-eps} and \ref{LEM: Frst d/n<r-1/r}.

\begin{proof}[Proof of Proposition \ref{LEM: Frst(a+1)>Frst(a) when a<1-1/r-eps}.]
We will only prove the case (i), and the case (ii) can be proved analogously.
Suppose that $t<s+\frac12+\sqrt{2s+\frac14}$.
Observe that in this case $H(\frac1{r-1})>0$.
We need to show $\Delta(a)>0$ for all $a\in[\gamma n,\lfloor\frac{r-1}{r}n\rfloor]$. Let $z=\frac{n-a}{a}$.
Then $z\in[\frac{1}{r-1},\frac{1-\gamma}{\gamma}]$. By Claims \ref{LEM: D=MH} and \ref{LEM: H is increasing},
it holds for sufficiently large $n$ that $\frac{\Delta(a)}{M(a)}=H(z)+o(1)\ge H\left(\frac1{r-1}\right)+o(1)>0.$
Since $M(a)>0$, this proves $\Delta(a)>0$.
\end{proof}

\begin{proof}[Proof of Proposition \ref{LEM: Frst d/n<r-1/r}.]
Let $\beta=\frac{r-1}{r}-\varepsilon$ and $\tilde C=\lambda_{s,t}\cdot C$. 	
By the definition of $\Delta(a)$ and Claim \ref{LEM: D=MH},
we see that $F\left(\left\lfloor\frac{r-1}{r}n\right\rfloor\right)-F\left(\left\lfloor\left(\frac{r-1}{r}-\varepsilon\right) n\right\rfloor\right)$ equals
\begin{equation}\label{equa: sum D}
\lambda_{s,t}\cdot\sum_{a=\lfloor\beta n\rfloor}^{\lfloor\frac{r-1}{r}n\rfloor-1}\Delta(a)=\lambda_{s,t}\cdot\sum_{a=\lfloor\beta n\rfloor}^{\lfloor\frac{r-1}{r}n\rfloor-1}M(a)\left[H\left(\frac{n-a}{a}\right)+o(1)\right],
\end{equation}
where $\lambda_{s,t}\cdot M(a)=\tilde C{a^{(r-2)s+t}(n-a)^{s-2}}.$
Then the equation \eqref{equa: sum D} becomes
\begin{equation*}
\tilde C\sum_{a=c n}^{\lfloor\frac{r-1}{r}n\rfloor-1}{a^{(r-2)s+t}(n-a)^{s-2}}H\left(\frac{n-a}{a}\right)+o(n^{(r-1)s+t-1}).\label{equa: sum a(n-a)}
\end{equation*}
We use Riemann integral to estimate the above summation as following
\begin{eqnarray*}
	&&\frac1{n^{(r-1)s+t-1}}\cdot\sum_{a=\lfloor\beta n\rfloor}^{\lfloor\frac{r-1}{r}n\rfloor-1}{a^{(r-2)s+t}(n-a)^{s-2}}H\left(\frac{n-a}{a}\right)\\
	&=&\sum_{a=\lfloor\beta n\rfloor}^{\lfloor\frac{r-1}{r}n\rfloor-1}\left[\left(\frac{a}{n}\right)^{(r-2)s+t}\left(\frac{n-a}{n}\right)^{s-2}H\left(\frac{n-a}{a}\right)\cdot\frac1n\right]\\
	&\xlongrightarrow{n\rightarrow\infty}&\int_{\beta}^{\frac{r-1}{r}}x^{(r-2)s+t}(1-x)^{s-2} H\left(\frac{1-x}x\right)dx =\int_{\frac1{r-1}}^{\frac1{\beta}-1}\frac{z^{s-2}H(z)}{(1+z)^{(r-1)s+t}}dz.
\end{eqnarray*}	
Let $I$ denote the above integral. By Claim \ref{LEM: H is increasing}, $H(z)>0$ for $z\in(\frac1{r-1},\frac1{\beta}-1)$. So $I>0$.
Putting everything together, one can obtain that
$F\left(\left\lfloor\frac{r-1}{r}n\right\rfloor\right)-F\left(\left\lfloor\left(\frac{r-1}{r}-\varepsilon\right)n\right\rfloor\right)=(\tilde CI+o(1))\cdot n^{(r-1)s+t-1}.$
Let $\xi=\frac{\tilde CI}2>0.$
Then it holds for sufficiently large $n$ that $F\left(\left\lfloor\frac{r-1}{r}n\right\rfloor\right)-F\left(\left\lfloor\left(\frac{r-1}{r}-\varepsilon\right) n\right\rfloor\right)>\xi n^{(r-1)s+t-1}.$
This proves Proposition \ref{LEM: Frst d/n<r-1/r}.
\end{proof}

\section{Proof of Proposition \ref{LEM: Key tech lem to pv complete case}}\label{app:C}
First we prove two claims. Let $q=t-s$ and $f(z)=(sz-q)(1+z)^q+(t+(r-2)s)z+q$.

\begin{claim}\label{PROP: H1+H2+H3=f+1/x}
There exists a polynomial $P(\alpha)$ with $P(0)=0$ such that the following holds.
For any fixed $C>0$, if $\alpha\in [0,C]$, then
$H(x,\alpha)=f(\alpha)+(q^2-q-rs+P(\alpha))/x+O(x^{-2}),$
where the absolute value of the constant term in $O(x^{-2})$ is bounded by $C, r, s$ and $t$.
\end{claim}
\begin{proof}
Recall that $H(x,\alpha)=\sum_{i=1}^3 H_i(x,\alpha)$. So we need to estimate each $H_i$.

Let $C>0$ be fixed and $\alpha\in [0,C]$.
Write $(z)_q=z^q+Az^{q-1}+g(z)$, where $g(z)$ is a polynomial of degree at most $q-2$.
Then we have ${(x+\alpha x)_q}=(1+\alpha)^qx^q+{A(1+\alpha)^{q-1}}x^{q-1}+O(x^{q-2}).$
From the definition of $H_1$ it follows that
$$H_1(x,\alpha)=\left(s\alpha-q-\frac{qs+t}{x}\right)\left(1+\frac{1-q}{x}\right)\left[(1+\alpha)^q+\frac{A(1+\alpha)^{q-1}}{x}+O(x^{-2})\right].$$
Expanding this multiplication, since $\alpha\in [0,C]$ is bounded, we obtain
$H_1(x,\alpha)=(s\alpha-q)(1+\alpha)^q+\tilde P_1(\alpha)/x+O(x^{-2}),$
where $\tilde P_1(\alpha)=-(qs+t)(1+\alpha)^q+(s\alpha-q)(1-q)(1+\alpha)^q+(s\alpha-q)A(1+\alpha)^{q-1}.$
Define $P_1(\alpha)=\tilde P_1(\alpha)-\tilde P_1(0)$, which is a polynomial with $P_1(0)=0$.
Then we have
\begin{equation*}
H_1(x,\alpha)=(s\alpha-q)(1+\alpha)^q+\frac{-(qs+t)-q(1-q)-qA+ P_1(\alpha)}{x}+O(x^{-2}).
\end{equation*}
By similar arguments one can write $H_2$ and $H_3$ as
\begin{eqnarray*}
H_2(x,\alpha)&=& t\alpha+q+\frac {qs-s+q+qA+P_2(\alpha)}x+O(x^{-2}),	\\
H_3(x,\alpha)&=& (r-2)s\alpha+\frac{-(r-2)s+P_3(\alpha)}{x}+O(x^{-2}),
\end{eqnarray*}
where $P_i$ is a polynomial with $P_i(0)=0$ for $i\in \{2,3\}$.
Summing up the above we obtain
$$H(x,\alpha)=f(\alpha)+\frac{q^2-q-rs+P(\alpha)}{x}+O(x^{-2}),$$
where $P(\alpha)=\sum_{i=1}^3 P_i(\alpha)$ is a polynomial with $P(0)=0$.
This proves Claim \ref{PROP: H1+H2+H3=f+1/x}.
\end{proof}

\begin{claim}\label{Prop: f is increasing}
The function $f(z)$ is strictly increasing in $[0,+\infty)$ with $f(0)=0$ and $f'(0)=sr+q-q^2$.
\end{claim}
\begin{proof}
It is easy to verify $f(0)=0$ and obtain
$$f'(z)=s(q+1)(1+z)^q-tq(1+z)^{q-1}+t+(r-2)s,$$
$$f''(z)=q(1+z)^{q-2}[s(q+1)z+2s+q-q^2].$$
So $f'(0)=sr+q-q^2.$	
Next we show $f(z)$ is strictly increasing.
If $q=0$, then $f(z)=sz+(t+(r-2)s)z=rsz$, which is obviously increasing.
So we may assume $q\ge 1$.
Since $s\leq t\leq s+\frac12+\sqrt{2s+\frac14}$, we have $2s+q-q^2\geq 0$.
This shows that $f'(0)=sr+q-q^2\geq 0$ and $f''(z)>0$ for $z>0$,
implying that $f'(z)>f'(0)\geq 0$ for $z>0$ and thus $f(z)$ is strictly increasing in $[0,+\infty)$. This completes the proof.	
\end{proof}

We are ready to prove Proposition \ref{LEM: Key tech lem to pv complete case}

\begin{proof}[Proof of Proposition \ref{LEM: Key tech lem to pv complete case}.]
First, we consider the case (i).
Suppose that $C>\varepsilon>0$ are fixed and $\alpha\in [\varepsilon, C]$.
By Claim \ref{PROP: H1+H2+H3=f+1/x} there exists a polynomial $P(\alpha)$ such that
$H(x,\alpha)=f(\alpha)+(q^2-q-rs+P(\alpha))/x+O(x^{-2}).$
By Claim \ref{Prop: f is increasing} we have $f(\alpha)\ge f(\varepsilon)>0.$
Since $|P(\alpha)|$ is bounded (as $\alpha\in [\varepsilon, C]$),
there exists a large $x_0>0$ such that for $x\ge x_0$
\begin{equation*}
\left|H(x,\alpha)-f(\alpha)\right|=\left|\frac{q^2-q-sr+P(\alpha)}{x}+O(1/x^2)\right|\le\frac12 f(\varepsilon).
\end{equation*}
Now it follows that $H(x,\alpha)\ge f(\varepsilon)-\frac12 f(\varepsilon)>0.$

Next we consider the case (ii). We have $s\leq t<s+\frac12+\sqrt{2s+\frac14}$,
which shows that $2s+q-q^2>0$. By Claim \ref{Prop: f is increasing}, $f(0)=0$ and $f'(0)=rs+q-q^2\geq 2s+q-q^2>0$.
So there exists $\varepsilon_1>0$ such that for $\alpha\in (0,\varepsilon_1]$ it holds that $|\frac{f(\alpha)}{\alpha}-f'(0)|\leq f'(0)/5$.
This implies that $f(\alpha)\geq 4f'(0)\alpha/5$ for $\alpha\in [0,\varepsilon_1]$.
Also since $P(\alpha)$ is a polynomial with $P(0)=0$, there exists $\varepsilon_2>0$ such that for $\alpha\in [0,\varepsilon_2]$,
$|P(\alpha)|\leq f'(0)/4.$
Applying Claim \ref{PROP: H1+H2+H3=f+1/x} with $C$ being $\varepsilon_0=\min\{\varepsilon_1,\varepsilon_2\}$, for $\alpha\in [0,\varepsilon_0]$ we have
$$H(x,\alpha)=f(\alpha)+\frac{P(\alpha)-f'(0)}{x}+O(x^{-2})\geq \frac{4f'(0)\alpha}{5}-\frac{5f'(0)}{4x}-\frac{D}{x^2},$$
where $D>0$ is bounded by $\varepsilon_0, r, s, t$.
Let $x_1=4D/f'(0)$. Then for $x\ge x_1$ and $\alpha\in [\frac2x,\varepsilon_0]$,
$$H(x,\alpha)\geq \frac{4f'(0)\alpha}{5}-\frac{5f'(0)}{4x}-\frac{f'(0)}{4x}=\left(\frac45\alpha-\frac{3}{2x}\right) f'(0)\geq \frac{f'(0)}{10x}>0.$$
This completes the proof of Proposition \ref{LEM: Key tech lem to pv complete case}.
\end{proof}

\end{document}